\documentclass[a4paper, 11pt, english]{article}
\usepackage[utf8]{inputenc}
\usepackage[T1]{fontenc}
\usepackage{graphicx}
\usepackage{stmaryrd}
\usepackage[a4paper]{geometry}
\usepackage{amsmath,amsfonts,amssymb,amsthm,epsfig,epstopdf,url,array}
\usepackage{rotating}
\usepackage[colorlinks=true,citecolor=red,linkcolor=blue,pdfpagetransition=Blinds]{hyperref}
\usepackage{cleveref}
\usepackage{nameref}
\usepackage{enumitem}
\usepackage{comment}
\Crefname{paragraph}{Section}{Sections}
\usepackage{fancyhdr}
\usepackage{xcolor}

\definecolor{by}{rgb}{0.74, 0.2, 0.64}

\usepackage{fullpage}
\crefname{theo}{theorem}{theorems}

\usepackage[sort,nocompress]{cite}

\newcommand{\ensemblenombre}[1]{\mathbb{#1}}
\newcommand{\N}{\ensemblenombre{N}}

\newcommand{\R}{} 
\renewcommand{\R}{\ensemblenombre{R}}

\newcommand{\norme}[1]{\left\lVert#1\right\rVert}

\newcommand{\intdoble}[3]{\int_{#1}^{#2}\!\!\!\int_{#3}}
\newcommand{\dive}[1]{\mathrm{div}}
\newcommand{\ov}[1]{\overline{#1}}

\providecommand{\keywords}[1]{\noindent {\textit{Keywords:}} #1}

\theoremstyle{plain} 
\newtheorem{prop}{Proposition}[section] 
\newtheorem{theo}[prop]{Theorem}

\newtheorem{lem}[prop]{Lemma}
\newtheorem{cor}[prop]{Corollary}
\theoremstyle{definition}
\newtheorem{defi}[prop]{Definition}
\newtheorem{rmk}[prop]{Remark}

\def\dx{\,\textnormal{d}x}
\def\dt{\textnormal{d}t}
\def\d{\textnormal{d}}
\def\esp{{\mathbb{E}}}
\def\dom{\mathcal{D}}
\def\fil{\mathcal{F}}

\makeatletter
\let\original@addcontentsline\addcontentsline
\newcommand{\dummy@addcontentsline}[3]{}
\newcommand{\DeactivateToc}{\let\addcontentsline\dummy@addcontentsline}
\newcommand{\ActivateToc}{\let\addcontentsline\original@addcontentsline}
\makeatother

\begin{document}

\title{Controllability results for stochastic coupled  systems of fourth- and second-order parabolic equations}
\author{V\'ictor Hern\'andez-Santamar\'ia \and Liliana Peralta}

\maketitle

\begin{abstract}
In this paper, we study some controllability and observability problems for stochastic systems coupling fourth- and second-order parabolic equations. The main goal is to control both equations with only one controller localized on the drift of the fourth-order equation. We analyze two cases: on one hand, we study the controllability of a linear backward system where the couplings are made only through first-order terms. The key point is to use suitable Carleman estimates for the heat equation and the fourth-order operator with the same weight to deduce an observability inequality for the adjoint system.  On the other hand, we study the controllability of a simplified nonlinear coupled model of forward equations. This case, which is well-known to be harder to solve, follows a methodology that has been introduced recently and relies on an adaptation of the well-known source term method in the stochastic setting together with a truncation procedure. This approach gives a new concept of controllability for stochastic systems. 
\end{abstract}
\small
\tableofcontents
\normalsize

\section{Introduction}\label{sec:intro}

The stabilized Kuramoto-Sivashinsky system was proposed in \cite{malomed} 
 as a model of front propagation in reaction-diffusion phenomena and combines dissipative features with dispersive ones. This system consists of a Kuramoto-Sivashinsky-KdV (KS-KdV) equation linearly coupled to an extra dissipative equation. More precisely, the model has the form
\begin{equation}\label{KS-H}
\begin{cases}
u_t  + \gamma u_{xxxx} + u_{xxx} + \lambda u_{xx} + uu_x =v_x, \\
v_t  - \Gamma v_{xx} + cv_x =u_x,
\end{cases}
\end{equation}
where  $\gamma,\lambda$ are positive coefficients accounting for the long-wave instability and the short-wave dissipation, respectively, $\Gamma>0$ is the dissipative parameter and $c\in\mathbb R\setminus\{0\}$ is the group-velocity mismatch between  wave modes.  

This model applies to the description of surface waves on multilayered liquid films and serves as a one-dimensional model for turbulence and wave propagation, see  \cite{malomed}  for a more detailed discussion.

The controllability of coupled systems like \eqref{KS-H} has attracted a lot of attention in the recent past (see for instance the survey \cite{AKBGBdT11} and the references within). One of the common features among such works is the goal of controlling as many equations with the fewest number of controls. Roughly speaking, it is more difficult to deduce control properties for coupled systems than for single equations. 

System \eqref{KS-H} has been studied from the controllability point of view in the deterministic setting in various papers. In \cite{cmp}, the authors address the boundary controllability when the control action is applied on both equations. Later, in \cite{cmp2}, it is proved that controllability with  a single control supported in an interior open subset of the domain and acting on the fourth-order equation can be achieved, while  \cite{cc16} studies  the analogous  problem but with an interior  control acting only on the heat equation. Finally, in \cite{CCM19}, the problem is addressed with a single boundary control but only for a linearization of the system. 

In this context, a natural question that arises is to what extent the controllability properties for the stochastic counterpart of \eqref{KS-H} hold. Seen individually, the fourth- and second-order equations that compose this system have been studied from the control point of view in several works. We refer to \cite{BRT03,TZ09,Liu14} for some of the most representative works about the controllability of stochastic heat and parabolic-type equations, while we refer to \cite{Gao15,Gao18} for some results about the linear stochastic fourth-order equation. Nevertheless, as a coupled system, to the best of the authors' knowledge, this kind of problem has not been studied yet in the literature.

\subsection{Statement of the main results}
In what follows, we fix $T>0$, $\dom=(0,1)$, and denote $\dom_0$ as any given nonempty open subset of $\dom$. We will denote $Q=\dom\times(0,T)$ and $\Sigma=\{0,1\}\times(0,T)$.

Let $(\Omega,\mathcal F, \{\mathcal F_t\}_{t\geq0},\mathbb P)$ be a complete filtered probability space on which a one-dimensional standard Brownian motion $\{W(t)\}_{t\geq 0}$ is defined such that $\{\mathcal F_t\}_{t\geq 0}$ is the natural filtration generated by $W(\cdot)$, augmented by all the $\mathbb P$-null sets in $\mathcal F$.  Let $X$ be a Banach space. For $p\in[1,\infty]$, we define the space
\begin{align*}
L^p_{\mathcal F}(0,T;X):=&\left\{\phi:\phi \text{ is an $X$-valued $\mathcal F_t$-adapted process}\right. \\
&\; \left.\text{on $[0,T]$, and $\phi\in L^p([0,T]\times\Omega;X)$}\right\},
\end{align*}
endowed with the canonical norm and we denote by $ L^2_{\mathcal F}(\Omega;C([0,T];X))$ the Banach space consisting of all $X$-valued $\mathcal F_t$-adapted continuous processes $\phi(\cdot)$ such that $E\left(\|\phi(\cdot)\|^2_{C([0,T];X)}\right)<\infty
$, also equipped with the canonical norm.

\subsubsection{Controllability of the backward system}

In the first part of the paper, we are interested in studying the null controllability for the following linear backward stochastic system
\begin{equation}\label{backward_KS-H}
\begin{cases}
\d y-(\gamma y_{xxxx}-y_{xxx}+y_{xx})\dt=(z_x-d_1 Y-d_2 Z+\chi_{\dom_0}h)\dt+Y\d W(t) &\text{in }Q, \\
\d z+\Gamma z_{xx}\dt=(z_x+y_x-d_3Z)\dt+Z\d W(t) &\text{in }Q, \\
y=y_x=0 &\text{on }\Sigma, \\
z=0 &\text{on }\Sigma, \\
y(T)=y_T, \quad z(T)=z_T &\text{in } \dom. 
\end{cases}
\end{equation}

In  \eqref{backward_KS-H}, $(y_T,z_T)$ is the terminal state, $(y,z)$ is the state variable, $h$ is the control variable, and $d_i$, $i=1,2,3$, are suitable coefficients. As it is common in the theory of BSDE, the additional processes $(Z,Y)$ are needed for the well-posedness of the system (see \Cref{prop:app_well_1}).

The controllability problem we are interested can be formulated as follows.  

\begin{defi}
System \eqref{backward_KS-H} is said to be null-controllable if for any given initial data $(y_T,z_T) \in L^2(\Omega,\mathcal F_T;L^2(\dom)^2)$, there exists a control $h\in L^2_\mathcal F(0,T;L^2(\dom_0))$ such that the solution $(y,z)$ to system \eqref{backward_KS-H} satisfies $(y(0),z(0))=0$ in $\mathcal D$, $P$-a.s.
\end{defi}

Observe that the control $h$ is applied only on the first equation of the system and acts indirectly through the coupling $y_x$ in the drift of the second equation. As we have mentioned before, this situation is more complicated than for a single equation and in the stochastic setting even more difficulties appear. Indeed, there are only a handful of works studying controllability problems for coupled stochastic systems with less controls than equations, see, \cite{LL12,Liu14b,LL18}. In particular, in \cite{LL18}, for controlling several parabolic equations with few controls, well-known facts such as Kalman-type conditions that are true in the deterministic setting (see e.g. \cite[Theorem 5.1]{AKBGBdT11}) are not longer valid for the stochastic setting. 

Moreover, looking at system \eqref{backward_KS-H}, we see that the equation is coupled by first-order terms only. This is of course more difficult than the case where only zero-order terms are used and classical methodologies for dealing with coupled systems are not longer valid (see e.g.  \cite[Theorem 1.2]{LL12} and \cite[Proposition 4]{LL18}). 

Here, inspired in the proof for the deterministic case  of \cite{cmp2}, we are able to prove that under suitable conditions for the coupling coefficients $d_i$, system \eqref{backward_KS-H} is null controllable. More precisely, we have the following.

\begin{theo}\label{main_teo}
Assume that $d_i\in L_{\mathcal F}^\infty(0,T;W^{2,\infty}(\dom))$ for $i=1,2$ and $d_3\in L^\infty_{\mathcal F}(0,T;\mathbb R)$. Then, system \eqref{backward_KS-H} is null-controllable.
\end{theo}

Using the classical equivalence between null-controllability and observability, the main ingredient of the proof consists in obtaining a suitable observability inequality for the corresponding adjoint system. Using stochastic versions of well-known Carleman estimates for the fourth-order operator and the heat equation with non-homogeneous Neumann boundary conditions, we adapt the methodology in \cite{cmp2} to the stochastic framework.

\subsubsection{Controllability of the forward system}

In the second part of the work, we are interested in studying the controllability of a simplified model of a stochastic version of \eqref{KS-H}. In more detail, let us consider the system 
\begin{equation}\label{forward_KS-H_nonlin}
\begin{cases}
\d y+(y_{xxxx}+ayy_x)\,\dt=\chi_{\dom_0}h\,\dt+(b_1y+ b_2 z)\, \d W(t) &\text{in }Q, \\
\d z- z_{xx}\,\dt=y\,\dt+b_3z\, \d W(t) &\text{in }Q, \\
y=y_{xx}=0 &\text{on }\Sigma, \\
z=0 &\text{on }\Sigma, \\
y(0)=y_0, \quad z(0)=z_0 &\text{in } \dom. 
\end{cases}
\end{equation}

In \eqref{forward_KS-H_nonlin}, $(y,z)$ is the state, $h$ is the control variable and $(y_0,z_0)$ is the initial datum. In this case, we assume that $b_i=b_i(t)\in L^\infty_{\mathcal{F}}(0,T;\mathbb R)$, $i=1,2,3$. Here, $a$ is either 0 or 1.

In the case where $a=0$, system \eqref{forward_KS-H_nonlin} is linear and its controllability is consequence of a general result shown in \cite{Liu14b}. In that work, the controllability of coupled fractional parabolic systems with one control force is addressed and under the assumptions on the coefficients $b_i$ and the boundary conditions above, the controllability of \eqref{forward_KS-H_nonlin} can be obtained just by setting $\alpha=2$ and $\beta=1$ in the notation of Theorem 1.1 of \cite{Liu14b}. The proof is achieved by means of the Lebeau-Robbiano technique and the obtention of a suitable observability inequality for the adjoint system is the main ingredient. 


In the case where $a=1$, system \eqref{forward_KS-H_nonlin} is clearly nonlinear and, as far as we know, its controllability has not been addressed in the literature. In fact, the controllability of nonlinear stochastic parabolic equations has been rarely studied in the literature. The main difficulty comes from the fact that, as pointed out in \cite{TZ09}, there is lack of compactness in the functional spaces where the equations are posed and well-known techniques in the deterministic setting cannot be readily translated to the stochastic framework.

In the recent works \cite{HSLBP20b,HSLBP20a}, two different approaches have been introduced to circumvent this lack of compactness and controllability results for semilinear stochastic parabolic equations have been established. The first approach, introduced in \cite{HSLBP20b}, uses refined Carleman estimates and a suitable fixed point argument to prove a global null-controllability for the heat equation with globally Lipschitz nonlinearities.

The second approach, introduced in \cite{HSLBP20a}, is inspired in the source term method of the seminal work \cite{LLT13} and the truncation approach used in \cite{Gao17} (see also \cite{Lia14}) and has been used to prove a new concept of controllability for the heat equation with polynomial-type nonlinearities. 
%
%
To make precise this new formulation, we begin by defining the weight function
\begin{equation*}
\hat{\rho}(t)=e^{-C/(T-t)}, \quad t\in[0,T)
\end{equation*}
where $C>0$ is a constant only depending on $\dom$, $\dom_0$ and the coefficients $b_i$, $i=1,2,3$. This constant is defined precisely at equation \eqref{eq:def_rho_hat_exact} below.

For every $t\in[0,T]$, we introduce the functional space
\begin{align*}
X_t :=\notag \Bigg\{ & (y,z) \in C([0,t];H^2(\dom)\times H_0^1(\dom)) \cap L^2(0,t;H^4(\dom)\times H^2(\dom))\ :\\ \notag
&   \Bigg( \sup_{0\leq s \leq t}\left\|\frac{y(s)}{\hat{\rho}(s)}\right\|^2_{H^2(\dom)} + \sup_{0\leq s \leq t}\left\|\frac{z(s)}{\hat{\rho}(s)}\right\|^2_{H_0^1(\dom)} \\
&\quad  + \int_{0}^{t}\left\|\frac{y(s)}{\hat{\rho}(s)}\right\|^2_{H^4(\dom)}\d{s}+\int_{0}^{t}\left\|\frac{z(s)}{\hat{\rho}(s)}\right\|^2_{H^2(\dom)}\d{s}\Bigg)^{1/2} < + \infty \Bigg\}, 
\end{align*}
endowed with its corresponding norm. Lastly, for each $R>0$, we define $\varphi_R\in C_0^\infty(\R^+)$ as
\begin{equation}\label{c_tr}
\varphi_R(s)=
\begin{cases}
1, &\mbox{if } s\leq R \\
0, &\mbox{if } s\geq 2R,
\end{cases}
\end{equation}
Notice that defined in this way
\begin{equation}\label{deriv_varphi_R}
\|\varphi_R^\prime\|_\infty\leq C/R.
\end{equation}

The main interest of introducing the above functions and spaces is that for establishing a controllability result for \eqref{forward_KS-H_nonlin}, first we will study the controllability of a simplified version of system  \eqref{forward_KS-H_nonlin} where the full nonlinearity is replaced by a truncated one. 

More precisely, for $f\in C^1(\R^2;\R)$, defined as $f(s,u)=s  u$ for all $(s,u)\in \R^2$, we will study the following system
\begin{equation}\label{forward_KS-H_nonlin_R}
\begin{cases}
\d y+(y_{xxxx}+f_R(y,y_x))\,\dt=\chi_{\dom_0}h\,\dt+(b_1y+ b_2 z)\, \d W(t) &\text{in }Q, \\
\d z- z_{xx}\,\dt=y\,\dt+b_3z\, \d W(t) &\text{in }Q, \\
y=y_{xx}=0 &\text{on }\Sigma, \\
z=0 &\text{on }\Sigma, \\
y(0)=y_0, \quad z(0)=z_0 &\text{in } \dom. 
\end{cases}
\end{equation}
where 
\begin{equation}\label{eq:def_fR}
f_R(y,y_x)=\varphi(\|{(y,z)}\|_{X_t}) f(y,y_x).
\end{equation}

Our first result is the following.

\begin{theo}
\label{th:mainresult1}
Let $T>0$. There exists a positive constant $C=C(\dom, \dom_0, b_1,b_2,b_3)$ such that for every $R \leq e^{-C/T}$, for every initial data $(y_0,z_0) \in L^2(\Omega,\fil_0;H^2(\dom)\cap H_0^1(\dom))\times L^2(\Omega,\fil_0; H_0^1(\dom)) $, there exists $h \in L^2_{\fil}(0,T;L^2(\dom))$, such that the solution $(y,z) \in L_{\mathcal{F}}^2(\Omega; C([0,T];H^2(\dom)))\times  L_{\mathcal{F}}^2(\Omega; C([0,T];H_0^1(\dom))) $ of \eqref{forward_KS-H_nonlin_R} satisfies 
\begin{equation}
\label{eq:estimatenonlinearityPropfR}
 \esp \left(\norme{(y,z)}_{X_T}^2 \right) 
\leq e^{2C/T}  \esp\left(\|y_0\|^2_{H^2(\dom)\cap H_0^1(\dom)}+\|z_0\|^2_{H_0^1(\dom)}\right) ,
\end{equation}
and
\begin{equation}
\label{eq:ynulth1}
y(T, \cdot) = 0\ \text{in}\ \dom,\ \text{a.s.}
\end{equation}
\end{theo}

This result is a small-time global controllability result for \eqref{forward_KS-H_nonlin_R}, that is, a \emph{truncated} version of the original system  \eqref{forward_KS-H_nonlin}. Indeed, notice that at any time $t\in[0,T]$, if the controlled solution to \eqref{forward_KS-H_nonlin_R} is such that $\|(y,z)\|_{X_t}>2R$, the corresponding nonlinearity $f_R$ vanishes. We emphasize here that this result is independent of the size of the initial data $(y_0,z_0)$.

Once this result is established, we can refine the theorem and provide the following small-time \emph{statistical} local null-controllability result. 

\begin{theo}
\label{th:mainresult2}
Let $\epsilon >0$ and $T>0$. Let $C=C(\dom, \dom_0, b_1,b_2,b_3) >0$ as in \Cref{th:mainresult1}, $R =  e^{-C/T}$ and $\delta = e^{-2C/T} \sqrt{\epsilon}$. Then, for every initial data $(y_0,z_0) \in L^2(\Omega,\fil_0;H^2(\dom)\cap H_0^1(\dom))\times L^2(\Omega,\fil_0; H_0^1(\dom)) $ verifying 
\begin{equation}\label{eq:small_data}
\norme{(y_0,z_0)}_{L^2(\Omega,\fil_0;H^2(\dom)\cap H_0^1(\dom)\times H_0^1(\dom))} \leq \delta,
\end{equation}
there exists $h \in L^2_\fil(0,T;L^2(\dom_0))$ such that the solution $(y,z) \in L_{\mathcal{F}}^2(\Omega; C([0,T];H^2(\dom)))\times  L_{\mathcal{F}}^2(\Omega; C([0,T];H_0^1(\dom))) $ of \eqref{forward_KS-H_nonlin_R} satisfies 
\begin{equation}
\label{eq:ynulth2}
y(T, \cdot) = 0\ \text{in}\ \dom,\ \text{a.s.}
\end{equation}
and
\begin{equation}
\label{eq:probafRf}
\mathbb{P} \bigg( \sup_{t \in [0,T]} \norme{(y,z)}_{X_t} \leq R \bigg) \geq 1 - \epsilon.
\end{equation}
\end{theo}

As in \cite{HSLBP20a}, we can justify the new terminology of \emph{statistical} null-controllability as follows. \Cref{th:mainresult2} states that given $T>0$ and a small parameter $\epsilon>0$, one is able to find a ball of size $\delta>0$ (depending on $T$ and $\epsilon$) such that, with a confidence level $1-\epsilon$, one can steer any initial data $(y_0,z_0)$ smaller than $\delta$ to 0 at time $T$ without crossing the threshold $R$. Actually, from \eqref{eq:probafRf} and \eqref{eq:def_fR}, we can deduce that the controlled trajectory constructed in \Cref{th:mainresult2} is a solution to the original system \eqref{forward_KS-H_nonlin} with probability greater than $1-\epsilon$. Also, as in \cite{HSLBP20a}, we emphasize that the existence and uniqueness of the nonlinear system \eqref{forward_KS-H_nonlin} is not needed a priori to study the control problem. In turn, our method yields the existence of a solution in the weighted space $X_t$ which has a controllability constraint built-in. 

\subsection{Outline of the paper}

The rest of the paper is organized as follows. In Section \ref{sec:proof_main}, we prove in a first part the observability inequality for a suitable forward adjoint system and then, following well-known arguments, we obtain the proof of Theorem \ref{main_teo}. In \Cref{sec:forward_cont}, we prove \Cref{th:mainresult1,th:mainresult2} by adapting the novel method introduced in \cite{HSLBP20a}. Finally, in Section \ref{sec:further}, we make some final comments about our work.

\section{Null controllability for the backward system}\label{sec:proof_main}

The goal of this section is to prove the null controllability of system \eqref{backward_KS-H}. As already mentioned, this problem will be reformulated in terms of the observability of the adjoint system, which in this case is given by
\begin{equation}\label{eq:adjoint_KSS}
\begin{cases}
\d u+(\gamma u_{xxxx}+u_{xxx}+u_{xx})\dt=v_xdt+d_1 u\, \d W(t) &\text{in }Q, \\
\d v-\Gamma v_{xx}\dt=(v_x+u_x)\dt+(d_2u+d_3 v) \d W(t) &\text{in }Q, \\
u=u_x=0 &\text{on }\Sigma, \\
v=0 &\text{on }\Sigma, \\
u(0)=u_0, \quad v(0)=v_0 &\text{in } \dom. 
\end{cases}
\end{equation}

The main task is reduced to prove the following result.
\begin{prop}\label{prop:observability_forward}
Under the assumptions of Theorem \ref{main_teo}, there exists a positive constant $C$ such that for every $(u_0,v_0)\in L^2(\Omega, \mathcal F_0;L^2(\dom)^2)$, the solution $(u,v)$ of \eqref{eq:adjoint_KSS} verifies 
\begin{equation}\label{eq:obs_forward}
\esp \left(\int_{\dom}\left(|u(T)|^2+|v(T)|^2\right)\dx\right)\leq C \esp \left(\int_{Q_{\dom_0}}|u|^2\dx\dt\right).
\end{equation}
\end{prop}
To prove \eqref{eq:obs_forward}, we use Carleman estimates in the spirit of \cite{cmp2} and adapt the procedure to the stochastic framework. As mentioned there, the key point is to have suitable estimates with the same weight functions allowing to obtain the observability inequality with only one observation term.

\subsection{Preliminaries on stochastic Carleman estimates}
We begin by recalling below two global Carleman estimates for the stochastic fourth-order equation and the stochastic heat equation with non-homogeneous boundary conditions. These will serve as the basis for obtaining our observability estimate. 

Let us consider some $\dom_1\subset\subset \dom_0$. Similar to \cite{fursi}, we show the following known result.	
\begin{lem}\label{lem:weight_fursi}
There is a $\psi\in C^\infty(\overline{\dom})$ such that $\psi>0$ in $\dom$, $\psi(0)=\psi(1)=0$ and $|\psi_x|>0$ in $\overline{\dom}\setminus \dom_1$. 
\end{lem}
Following \cite{Gue07}, for some positive constants $k, m, \mu$, where $k>m$ and $m>3$, we define
\begin{equation}\label{eq:weights}
\alpha_m(x,t)=\frac{e^{\mu(\psi(x)+c_2)}-e^{\mu c_1}}{t^m(T-t)^m}, \quad \phi_m(x,t)=\frac{e^{\mu(c_2+\psi(x))}}{t^m(T-t)^m}
\end{equation}
where
\begin{equation*}
c_1=k\left(\tfrac{m+1}{m}\right)\|\psi\|_{\infty} \quad\text{and}\quad c_2=k\|\psi\|_\infty.
\end{equation*}

We define the weights in this way to fulfill the requirement that both estimates must have the same weight (cf. \cite{cmp2}).

In the remainder of this section, we set $\theta=e^{\lambda \alpha_m}$ and in order to abridge the estimates, we use the following notation
\begin{gather}\label{eq:ab_rhs_KS}
I_{KS}(p):=\esp\int_{Q}\theta^2\lambda\phi_m\left(|p_{xxx}|^2+\lambda^2\phi_m^2|p_{xx}|^2+\lambda^4\phi_m^4|p_x|^2+\lambda^6\phi^6_m|p|^2\right)\dx\dt, \\
I_{H}(q):=\esp\int_{Q}\theta^2\lambda\phi_m\left[|q_x|^2+\lambda^2\phi_m^2|q|^2\right]\dx\dt.
\end{gather}

The Carleman inequality for the forward stochastic KS system we shall use reads as follows.
\begin{lem}\label{lem:carleman_KS}
Let  $f\in L^2_{\mathcal F}(0,T;L^2(\dom))$ and $F\in L^2_{\mathcal F}(0,T;H^2(\dom))$ be given. There exist positive constants $C$, $\mu_1$ and $\lambda_0$, such that for any $\mu\geq \mu_0$, any $\lambda\geq\lambda_0(\mu)$ and any $p_0\in L^2(\Omega;\mathcal F_0; L^2(\dom))$, the solution $p$ to
\begin{equation*}
\begin{cases}
\d p +p_{xxxx}\dt=f\dt +F \d W(t) &\text{in }Q, \\
p=p_x=0 &\text{in }\Sigma, \\
p(x,0)=p_0 &\text{in } \dom,
\end{cases}
\end{equation*}
satisfies
\begin{equation}\label{eq:car_KS}
\begin{split}
I_{KS}(p)\leq C& \esp \left( \int_{Q}\theta^2 |f|^2\dx\dt+\int_{Q}\theta^2\lambda^4\phi_m^4|F|^2\dx\dt+\int_{Q}\theta^2\lambda^2\phi_m^2|F_x|^2\dx\dt \right. \\
&\quad \left. +\int_{Q}\theta^2 |F_{xx}|^2 \dx\dt +  \int_{Q_{\dom_0}} \theta^2\lambda^7\phi_m^7|p|^2 \dx\dt \right).
\end{split}
\end{equation}
\end{lem}
The proof of \Cref{lem:carleman_KS} is essentially given in \cite{Gao15}. Actually, the authors prove inequality \eqref{eq:car_KS} for slightly different weight functions, that is, they take $m=1$, $c_2=3$ and $c_1=5$ in \eqref{eq:weights}. However, a closer inspection shows that their proof can be adapted to our case just by considering that $|\partial_t\alpha_{m}|\leq CT\phi^{1+\frac{1}{m}}$ and $|\partial_{tt}\alpha_{m}|\leq CT^2\phi^{1+\frac2m}$ and changing a little bit the estimates regarding $\alpha_{xt}, \alpha_{xxt}$, $\alpha_{xxxt}$ and so on in \cite[pp. 487]{Gao15}. It is also important to mention that we have kept the term containing $p_{xxx}$ in the left-hand side of \eqref{eq:car_KS} (see eq. \eqref{eq:ab_rhs_KS}) as it will be useful later. 

On the other hand, we have the following inequality for the heat equation with non-homogeneous boundary conditions. 

\begin{lem}\label{lem:carleman_H}
Let $g_1, G\in L^2_{\mathcal F}(0,T;L^2(\dom))$ and $g_2\in L^2_{\mathcal F}(0,T;L^2(\partial \dom))$ be given. There exist positive constants $\mu_1$, $\lambda_1$ and $C$, such that for any $\mu\geq \mu_1$, any $\lambda\geq \lambda_1(\mu)$ and any $q_0\in L^2(\Omega,\mathcal F_0;L^2(\dom))$, the solution $q$ to 
\begin{equation*}
\begin{cases}
\d q-q_{xx}\dt=g_1\dt+G\, \d W(t) &\textnormal{in }Q, \\
q_x=g_2 &\textnormal{on } \Sigma, \\
q(x,0)=q_0 &\textnormal{in } \dom,
\end{cases}
\end{equation*}
satisfies
\begin{equation}\label{eq:car_H}
\begin{split}
I_{H}(q)\leq C&\esp\left( \int_{\Sigma}\theta^2\lambda\phi_m|g_2|^2\d \sigma\dt+ \int_{Q_{\dom_0}}\theta^2\lambda^3\phi_m^3|q|^2\dx\dt \right. \\ 
&\quad \left. +  \int_{Q}\theta^2\left(|g_1|^2+\lambda^2\phi_m^2|G|^2\right)\dx\dt \right).
\end{split}
\end{equation}
\end{lem}
The proof of \Cref{lem:carleman_H} can be found in \cite{yan18}. As for the case of the KS equation, the proof can be adapted by making the corresponding changes. Notice that in this case, there are not spatial derivatives of the diffusion term $G$ on the right-hand side of estimate \eqref{eq:car_H}. This comes from the fact that the proof is done by means of a duality argument, instead of a pointwise estimate technique as used, for instance, in \cite{TZ09}.

\subsection{Carleman estimate for the adjoint system}

Now, we are in position to prove the main result of this section, that is to say, a Carleman estimate for system \eqref{eq:adjoint_KSS} with only one observation on the right-hand side. In more detail, we have the following result.
\begin{theo}\label{thm:car_estimate_KSS}
Assume that $d_i\in L^2_{\mathcal F}(0,T;W^{2,\infty}(\dom))$ for $i=1,2$, $d_3\in L^\infty_{\mathcal F}(0,T;\mathbb R)$ and let $m>3$ be given. Then, there exists $C>0$ and two constants $\mu_2,\lambda_2>0$ such that for any $\mu\geq \mu_2$, $\lambda\geq \lambda_2$ and any $(u_0,v_0)\in L^2(\Omega,\mathcal F_0;L^2(\Omega)^2)$, the solution $(u,v)$ to \eqref{eq:adjoint_KSS} verifies
\begin{equation}\label{eq:car_final}
\esp\left(\int_{Q}\theta^2\lambda^3\phi_m^3|v_x|^2\dx\dt+\int_{Q}\theta^2\lambda^7\phi^7_m|u|^2\dx\dt\right) \leq {C\esp\left(\int_{Q_{\dom_0}}\theta^2\lambda^{23}\phi_m^{23}|u|^2\dx\dt\right)}.
\end{equation}
\end{theo}

The outline of the proof follows the strategy of \cite{cmp2}, but since we are dealing with stochastic PDEs, some extra arguments are needed to conclude. For clarity, we have divided the proof in four steps which can be summarized as follows:
\begin{itemize}
\item First part: we look for the equation satisfied by $v_x$. As we will see, this equation has not prescribed boundary conditions, but we can apply estimate \eqref{eq:car_H} to deduce an inequality with some boundary terms and a local estimate of $v_x$. Using trace and interpolation estimates, we can get rid of the boundary terms. 
\item Second part: we apply Carleman inequality \eqref{eq:car_KS} to the first equation of system \eqref{eq:adjoint_KSS} and add it to the estimate in the previous step. By using the variable $\lambda$ we will absorb the lower order terms.
\item Third part: here, we will estimate the local term of $v_x$ by using the first equation \eqref{eq:adjoint_KSS} and leading to several local terms depending on $u$ and their spatial derivatives.
\item Fourth part: this step is divided in two: first, using the equation verified by $u$, we estimate the local term corresponding to the the highest-order derivative coming from the previous stage. Then, integrating by parts several times, we deduce the desired result.
\end{itemize} 

\begin{proof}[Proof of Theorem \ref{thm:car_estimate_KSS}]
Consider sets $\dom_i\subset\dom$, $i=4,5$, such that $\dom_{5}\subset\subset\dom_4\subset\subset \dom_0$. In the following, $C$ stands for a generic positive constant that may vary from line to line.  {We also consider that the initial datum $(u_0,v_0)$ is smooth enough (as usual, the general case follows from a density argument)}.

\subsubsection*{Step 1. Carleman estimate for $v_x$}
A direct computation shows that $v_x$ verifies the equation
\begin{equation*}
\d v_x-\Gamma (v_x)_{xx}\dt=[u_{xx}+(v_x)_x]\dt+[(d_2u)_x+d_3 v_x]\d W(t) \quad\text{in }Q,
\end{equation*}
with no prescribed boundary conditions. Hence, applying estimate \eqref{eq:car_H} (for the set $\dom_4$) to this equation yields
\begin{align}\notag
I_H(v_x)\leq C&\esp \left(\int_{\Sigma}\theta^2\lambda\phi_m|v_{xx}|^2\d\sigma\dt+\int_{Q_{\dom_4}}\theta^2\lambda^3\phi^3_m|v_x|^2\dx\dt\right) \\ \label{eq:car_init}
&\quad +C\esp\left(\int_{Q} \theta^2|u_{xx}+v_{xx}|^2\dx\dt+\int_{Q} \theta^2\lambda^2\phi^2_m|(d_2u)_x+d_3 v_x|^2\dx\dt\right).
\end{align}
Observe that 
\begin{align*}
\esp \left(\int_{\Sigma}\theta^2\lambda \phi_m|v_{xx}|^2\d\sigma\dt\right)&= \esp\left(\int_{0}^{T}\theta^2(1,t)\lambda\phi_m(1,t)|v_{xx}(1,t)|^2\dt\right) \\
&\quad -\esp\left( \int_{0}^{T}\theta^2(0,t)\lambda\phi_m(0,t)|v_{xx}(0,t)|^2\dt\right).
\end{align*}
For compactness, we will maintain the notation in \eqref{eq:car_init}.

Using that $d_3\in L^\infty_{\mathcal F}(0,T;\mathbb R)$ and taking $\lambda_1$ large enough, we can simplify the above expression as follows
\begin{equation}\label{eq:est_inter_1}
\begin{split}
I_H(v_x)\leq &C\esp\left(\int_{\Sigma}\theta^2\lambda\phi_m|v_{xx}|^2\d\sigma\dt+\int_{Q_{\dom_4}}\theta^2\lambda^3\phi^3_m|v_x|^2\dx\dt\right) \\
&+ C\esp\left(\int_{Q} \theta^2|u_{xx}|^2\dx\dt+\int_{Q} \theta^2\lambda^2\phi^2_m|(d_2u)_x|^2\dx\dt\right).
\end{split}
\end{equation}
for any $\lambda\geq \lambda_1$. 

Let us focus now on the first term in the right-hand side of the previous inequality. Observe that thanks to the properties of the weight function $\psi$ (see \Cref{lem:weight_fursi}), the functions $\theta$ and $\phi_m$ achieve its minimum at the boundary points, that is
\begin{align} \label{def_phi_star}
&\phi_m(0,t)=\phi_m(1,t)=\min_{x\in\overline\dom}\phi_m=:\phi_m^\star(t), \\ \label{def_theta_star}
&\theta(0,t)=\theta(1,t)=\min_{x\in\overline\dom}\theta=:\theta^\star(t).
\end{align}
Using this notation and employing classical trace and Sobolev embedding theorems, we can obtain that
\begin{equation*}
\esp\left(\int_{\Sigma}(\theta^\star)^2\lambda\phi_m^\star|v_{xx}|^2\d\sigma\dt\right) \leq C\esp\left(\int_{0}^{T}(\theta^\star)^2\lambda\phi_m^\star\|v\|_{H^{5/2+\epsilon}(\dom)}^2\dt\right), \quad \forall \epsilon>0.
\end{equation*}
Let us fix $0<\epsilon<1/2$. Applying the classical interpolation inequality in Sobolev spaces
\begin{equation*}
\|v\|_{H^{t_\sigma}}\leq \|v\|_{H^{t_0}}^{1-\sigma}\, \|v\|^{\sigma}_{H^{t_1}},
\end{equation*}
where $t_0,t_1\in\mathbb R$ and $t_\sigma=(1-\sigma)t_0+\sigma t_1$, $\sigma\in[0,1]$, with $t_0=3$ and $t_1=1$ we get
\begin{equation*}
\esp\left(\int_{\Sigma}(\theta^\star)^2\lambda\phi_m^\star|v_{xx}|^2\d\sigma\dt\right) \leq C \esp \left(\int_{0}^{T}(\theta^\star)^2\lambda\phi_m^\star\left(\|v\|_{H^3(\dom)}^{1-\sigma}\|v\|^\sigma_{H^{1}(\dom)}\right)^2\dt\right)
\end{equation*}
for some $\sigma=\sigma(\epsilon)\in (0,\frac{1}{4})$. 

We can conveniently rewrite the right-hand side of the above expression as
\begin{equation*}
\esp\left(\int_{0}^{T}\left[(\theta^\star)^{2\sigma}(\lambda\phi_m^\star)^{3\sigma}\|v\|^{2\sigma}_{H^1(\dom)}\right]\left[(\theta^\star)^{2-2\sigma}(\lambda\phi_m^\star)^{1-3\sigma}\|v\|_{H^3(\dom)}^{2(1-\sigma)}\right]\dt\right)
\end{equation*}
and using Young inequality with $p=1/\sigma$ and $q=1/(1-\sigma)$, we get for any $\delta>0$ 
\begin{align}\notag
\esp&\left(\int_{\Sigma}(\theta^\star)^2\lambda\phi_m^\star|v_{xx}|^2\d\sigma\dt\right) \\ \notag
&\leq \delta \esp\left(\int_{0}^{T}(\theta^\star)^2(\lambda\phi_m^\star)^3\|v\|^2_{H^1(\dom)}\dt\right)+C_\delta\esp\left(\int_{0}^{T}(\theta^\star)^2(\lambda\phi_m^\star)^{\frac{1-3\sigma}{1-\sigma}}\|v\|^2_{H^3(\dom)}\dt\right) \\ \label{eq:est_boundary}
&\leq \delta \esp\left(\int_{Q}\theta^2\lambda^3\phi_m^3|v_x|^2\dx\dt\right)+C_\delta\esp\left(\int_{0}^{T}(\theta^\star)^2(\lambda\phi_m^\star)^{\frac{1-3\sigma}{1-\sigma}}\|v\|^2_{H^3(\dom)}\dt\right),
\end{align}
where we have used the fact that $v=0$ on $\Sigma$ and definitions \eqref{def_phi_star}--\eqref{def_theta_star}.

Using estimate \eqref{eq:est_boundary} in \eqref{eq:est_inter_1} and taking $\delta>0$ small enough yields
\begin{equation}\label{eq:est_inter_1bis}
\begin{split}
I_H(v_x)\leq C&\esp \left ( \int_{0}^{T}(\theta^\star)^2(\lambda\phi_m^\star)^{\frac{1-3\sigma}{1-\sigma}}\|v\|^2_{H^3(\dom)}\dt+\int_{Q_{\dom_4}}\theta^2\lambda^3\phi^3_m|v_x|^2\dx\dt\right) \\
&\quad+ C\esp\left(\int_{Q} \theta^2|u_{xx}|^2\dx\dt+\esp\int_{Q} \theta^2\lambda^2\phi^2_m|(d_2u)_x|^2\dx\dt\right).
\end{split}
\end{equation}
Now, the task is to absorb the global term containing the $H^3$-norm. To this end, consider the function
\begin{equation*}
g(t)=\lambda^{\frac{1}{2}-\frac{1}{m}}\theta^\star(\phi_m^\star)^{\frac{1}{2}-\frac{1}{m}}
\end{equation*}
and define the change of variables $\widetilde v:=g(t)v$. Then, using It\^{o}'s formula, we see that $\widetilde v$ satisfies the equation
\begin{equation*}
\begin{cases}
\d \widetilde{v}=(\Gamma \widetilde{v}_{xx}+\widetilde{v}_x+g_t v+gu_x)\dt+(d_2g u+ d_3 \widetilde{v})\d W(t) &\text{in }Q, \\
\widetilde{v}=0 &\text{on } \Sigma, \\
\widetilde{v}(0)=0 &\text{in }\dom,
\end{cases} 
\end{equation*}
where we have used that $\lim_{t\to 0} g(t)=0$. Using classical energy estimates for stochastic parabolic equations (see, for instance, \cite[Proposition 2.1]{zhou92}), we have that $\widetilde v$ satisfies 
\begin{align}\notag
\sup_{0\leq t\leq T}&\esp\left(\|\widetilde{v}(t)\|^2_{H^2(\dom)}\right)+\esp\left(\int_{0}^{T}\|\widetilde{v}(t)\|^2_{H^3(\dom)}\dt\right) \\ \label{eq:est_H3}
&\leq C\esp \left(\int_{0}^{T}\|g_t v+g u_x\|_{H^1(\dom)}^2\dt+ \int_{0}^{T}\|d_2 g u\|_{H^2(\dom)}^2\dt\right).
\end{align}
We observe that if we choose $m>\frac{1-\sigma}{\sigma}$ then $\frac{1-3\sigma}{1-\sigma}<1-\frac{2}{m}$, thus, from the definition of $\widetilde{v}$, we obtain
\begin{equation*}
\esp\left(\int_{0}^{T}(\lambda\phi^\star_m)^{\frac{1-3\sigma}{1-\sigma}}(\theta^\star)^2\|v\|^2_{H^3(\dom)}\dt\right)\leq \esp \left(\int_{0}^{T}\|\widetilde v\|^2_{H^3(\dom)}\dt\right)
\end{equation*}
whence we have from \eqref{eq:est_H3}
\begin{align}\label{eq:est_H3_2}
\esp\left(\int_{0}^{T}(\lambda\phi_m^\star)^{\frac{1-3\sigma}{1-\sigma}}(\theta^\star)^2\|v\|^2_{H^3(\dom)}\dt\right)
 \leq C\esp \left(\int_{0}^{T}\|g_t v+g u_x\|_{H^1(\dom)}^2\dt+ \int_{0}^{T}\|d_2 g u\|_{H^2(\dom)}^2\dt\right).
\end{align}
We remark that choosing $m>\frac{1-\sigma}{\sigma}$ with $\sigma\in(0,\frac{1}{4})$ implies that $m>3$ which is consistent with the construction of the weights \eqref{eq:weights}.

Using once again that $v$ has homogeneous Dirichlet boundary conditions and since
\begin{equation}\label{eq:est_deriv_weight}
\left|\partial_t\left(\theta^\star[\phi_m^\star]^q\right)\right|\leq C\lambda(\phi_m^\star)^{1+\frac{1}{m}}(\theta^\star[\phi_m^\star]^q),
\end{equation}
for any integer $q>0$, we can bound in the right-hand side of \eqref{eq:est_H3_2} as follows
\begin{align*}\notag
\esp&\left(\int_{0}^{T}(\lambda\phi_m^\star)^{\frac{1-3\sigma}{1-\sigma}}(\theta^\star)^2\|v\|^2_{H^3(\dom)}\dt\right) \\ \notag
 &\leq C\esp \left(\int_{0}^{T}\lambda^{3-\frac{2}{m}}(\phi_m^\star)^{3}(\theta^\star)^2\|v_x(t)\|_{L^2(\dom)}^2 \dt +\int_{0}^{T}\|g u_x\|_{H^1(\dom)}^2\dt+ \int_{0}^{T}\|d_2g u\|_{H^2(\dom)}^2\dt\right) \\
&\leq C\esp \left(\int_Q \lambda^{3-\frac{2}{m}}\phi_m^{3}\theta^2|v_x|^2\dx\dt +\int_{0}^{T}\|g u_x\|_{H^1(\dom)}^2\dt+ \int_{0}^{T}\|d_2g u\|_{H^2(\dom)}^2\dt\right),
\end{align*}
where we have recalled \eqref{def_phi_star}--\eqref{def_theta_star}. 

For the last two terms in the above expression, it can be readily seen that since $u\in H^2_0(\dom)$ and $d_2\in L^\infty_{\mathcal F}(0,T;W^{2,\infty}(\dom))$, we have
\begin{align}\notag 
\esp&\left(\int_{0}^{T}(\lambda\phi_m^\star)^{\frac{1-3\sigma}{1-\sigma}}(\theta^\star)^2\|v\|^2_{H^3(\dom)}\dt\right)\leq C\esp \left(\int_Q \lambda^{3-2/m}\phi_m^{3}\theta^2|v_x|^2\dx\dt\right) \\ \label{eq:est_defi_H3} 
&+ C\|d_2\|^2_{L^\infty_{\mathcal{F}}(0,T;W^{2,\infty}(\dom))}\esp \left( \int_Q \lambda^{1-2/m}\theta^2\phi_m^{1-2/m}|u_{xx}|^2\dx\dt \right).
\end{align}
Observe that the power of $\lambda$ in the first term of the right-hand side is lower than its counterpart in the left-hand side of \eqref{eq:est_inter_1bis}. Hence, combining estimates \eqref{eq:est_defi_H3} and \eqref{eq:est_inter_1bis} and taking $\lambda$ large enough, we get
\begin{equation}\label{eq:est_heat_fin}
\begin{split}
I_H(v_x)\leq C\esp&\left(\int_{Q_{\dom_4}}\theta^2\lambda^3\phi^3_m|v_x|^2\dx\dt + \int_{Q} \lambda^{1-2/m}\theta^2\phi_m^{1-2/m}|u_{xx}|^2\dx\dt \right. \\
&\quad \left.+\int_{Q}\theta^2|u_{xx}|^2\dx\dt+\int_{Q} \theta^2\lambda^2\phi^2_m|(d_2u)_x|^2\dx\dt\right),
\end{split}
\end{equation}
for any $\lambda\geq \lambda_1$.

\subsubsection*{Step 2. Carleman estimate for $u$}
Fix $m>3$ coming from the previous step. We apply inequality \eqref{eq:car_KS} to the first equation of system \eqref{eq:adjoint_KSS}, note that both estimates have the same weight. We readily see that
\begin{align*}\notag
I_{KS}(u)\leq C &\esp\left(\int_{Q}\theta^2|v_x-u_{xx}-u_{xxx}|^2\dx\dt+\int_{Q}\theta^2\lambda^4\phi_m^4|d_1u|^2\dx\dt \right. \\
&\quad \left.+\int_{Q}\theta^2\lambda^2\phi_m^2\left|(d_1 u)_x\right|^2 +\int_{Q}\theta^2\left|(d_1 u)_{xx}\right|^2\dx\dt+\int_{Q_{\dom_4}} \theta^2\lambda^7\phi_m^7|u|^2\dx\dt\right).
\end{align*}
Using that $d_1\in L^\infty_{\mathcal F}(0,T;W^{2,\infty}(\dom))$ and taking $\lambda\geq \lambda_0$ large enough, we can absorb the lower order terms corresponding to the variable $u$, more precisely, 
\begin{align}\label{eq:est_ks_fin}
I_{KS}(u)\leq C &\esp\left(\int_{Q}\theta^2|v_x|^2\dx\dt+\int_{Q_{\dom_4}} \theta^2\lambda^7\phi_m^7|u|^2\dx\dt\right).
\end{align}

Adding up inequalities \eqref{eq:est_heat_fin} and \eqref{eq:est_ks_fin}, we take $\mu\geq \mu_2=\max\{\mu_0,\mu_1\}$ and $\lambda\geq \lambda_2=\max\{\lambda_0,\lambda_1\}$, we can absorb the remaining lower order terms to finally obtain
\begin{equation}\label{eq:car_locales_u_v}
I_H(v_x)+I_{KS}(u)\leq C\esp \left(\int_{Q_{\dom_4}}\theta^2\lambda^3\phi_m^3|v_x|^2\dx\dt+\int_{Q_{\dom_4}}\theta^2\lambda^7\phi_m^7|u|^2\dx\dt\right)
\end{equation}
for all $\lambda$ sufficiently large.

\subsubsection*{Step 3. Local energy estimate for $v$}
The main goal of this step is to estimate the local term corresponding to $v_x$. We follow the classical methodology introduced in \cite{deT00} but adapted to the stochastic setting.

Let us consider an open set $\dom_3$ such that $\dom_4\subset\subset\dom_3\subset\subset\dom_0$ and take a function $\eta\in C_0^\infty(\dom_3)$ such that $\eta\equiv 1$ in $\dom_4$. We define $\zeta:=\lambda^3\phi_m^3\theta^2\eta$ and apply It\^{o}'s formula to compute $\d(\zeta u v_x)$, 
%
%
from which we deduce
\begin{align}\notag 
\esp\left(\int_{Q}\zeta |v_x|^2\dx\dt\right)&=-\esp\left(\int_{Q}\zeta_t u_{v_x}\dx\dt\right)+\esp\left(\int_Q \zeta(\gamma v_xu_{xxxx}+v_xu_{xxx}+v_{x}u_{xx})\dx\dt\right) \\ \notag
&\quad - \esp\left(\int_Q\zeta(\Gamma uv_{xxx}+uu_{xx}+uv_{xx})\dx\dt\right) \\ \notag
&\quad - \esp\left(\int_Q \zeta[d_1u (d_2 u)_x+d_{1}u\,d_3v_x]\dx\dt\right) \\ \label{eq:iden_zeta_v}
&=: \sum_{i=1}^{9} K_i. 
\end{align}

Let us estimate each $K_i$, $1\leq i\leq 9$. For $i=1$, we can use \eqref{eq:est_deriv_weight} (which is also valid for $\theta$ and $\phi_m$) and Cauchy-Schwarz and Young inequalities to obtain
\begin{equation}\label{eq:est_K1}
|K_1|\leq \epsilon \esp\left(\int_{Q}\eta\lambda^3\phi_m^3\theta^2|v_x|^2\dx\dt\right) + C_\epsilon \esp\left(\int_{Q}\eta\lambda^5\phi_m^{5+\frac{2}{m}}\theta^2|u|^2\dx\dt\right)
\end{equation}
for any $\epsilon>0$.

Integrating by parts in the space variable, we have
\begin{align}\notag
K_2&=-\esp\left(\int_{Q}\gamma \eta \lambda^3\phi_m^3\theta^2v_{xx}u_{xxx}\dx\dt\right)-\esp\left(\int_{Q}\gamma \lambda^3\phi_m^3\theta^2v_x\eta_xu_{xxx}\dx\dt\right) \\ \label{eq:est_inter_K2}
&\quad -\esp\left(\int_{Q}\gamma \lambda^3\eta v_x(\phi_m^3\theta^2)_{x}u_{xxx}\dx\dt\right).
\end{align}
Noting that
\begin{equation}\label{eq:deriv_x_weight}
|\partial_x(\theta^2\phi_m^q)|\leq C\lambda\phi_m\theta^2\phi_m^q, \quad \forall q\in \mathbb Z+,
\end{equation}
it is not difficult to see that
\begin{align}\notag
|K_2| &\leq \delta  \esp\left(\int_{Q}\theta^2\lambda\phi_m |v_{xx}|^2\dx\dt\right)+2\epsilon \esp\left(\int_{Q}\theta^2\lambda^3\phi_m^3|v_x|^2\dx\dt\right) \\ \label{eq:est_K2}
&\quad +C_{\epsilon,\delta}\left(\esp\int_{Q_{\dom_3}}\theta^2\lambda^5\phi_m^5|u_{xxx}|^2\dx\dt\right)
\end{align}
for any $\delta,\epsilon>0$. In this part, we have used that $\textnormal{supp } \eta_x\subset \dom_3$ for estimating the second term of \eqref{eq:est_inter_K2}.

For the third term, we have from Cauchy-Schwarz and Young inequalities
\begin{equation}\label{eq:est_K3}
|K_3|\leq \epsilon \esp\left( \int_{Q}\eta \lambda^3\phi_m^3\theta^2|v_x|^2\dx\dt\right)+C_{\epsilon}\esp\left(\int_{Q}\lambda^3\phi_m^3\theta^2\eta|u_{xxx}|^2\dx\dt\right).
\end{equation}
In the same fashion, we easily have
\begin{equation}\label{eq:est_K4}
|K_4|\leq \epsilon \esp\left( \int_{Q}\eta\lambda^3\phi_m^3\theta^2|v_x|^2\dx\dt\right)+C_{\epsilon}\esp\left(\int_{Q}\eta\lambda^3\phi_m^3\theta^2|u_{xx}|^2\dx\dt\right).
\end{equation}

For the term $K_5$, we integrate by parts in the space variable to get
\begin{align*}
K_5&=\esp\left(\int_{Q}\Gamma \lambda^3\phi_m^3\theta^2\eta u_{x}v_{xx}\dx\dt\right)+\esp\left(\int_{Q}\Gamma \lambda^3\phi_m^3\theta^2\eta_xu v_{xx}\dx\dt\right) \\
&\quad + \esp\left(\int_{Q}\Gamma \lambda^3\eta(\phi_m^3\theta^2)_xu v_{xx}\dx\dt\right).
\end{align*}
Using \eqref{eq:deriv_x_weight} and the properties of the function $\eta$, we get after succesive application of Cauchy-Schwarz and Young inequalities
\begin{align}\notag
|K_5|&\leq 3\esp\left(\delta \int_{Q}\lambda\phi_m\theta^2|v_{xx}|^2\dx\dt\right)+C_{\delta}\esp\left(\int_{Q_{\dom_3}}\lambda^7\phi_m^7\theta^2|u|^2\dx\dt\right) \\ \label{eq:est_K5}
&\quad + C_{\delta}\esp \left(\int_{Q_{\dom_3}}\lambda^5\phi_m^5\theta^2|u_x|^2\dx\dt\right)
\end{align}
for any $\delta>0$. 

For the sixth and seventh terms, we have
\begin{equation}\label{eq:est_K6}
|K_6|\leq \frac{1}{2}\esp\left(\int_{Q}\eta\lambda^3\phi_m^3\theta^2|u|^2\dx\dt\right)+\frac{1}{2}\esp\left(\int_{Q}\eta\lambda^3\phi_m^3\theta^2|u_{xx}|^2\dx\dt\right),
\end{equation}
and
\begin{equation}\label{eq:est_K7}
|K_7|\leq \delta \esp\left(\int_{Q}\eta\lambda\phi_m\theta^2|v_{xx}|^2\dx\dt\right)+C_{\delta}\esp\left(\int_{Q}\eta\lambda^5\phi_m^5\theta^2|u|^2\dx\dt\right).
\end{equation}

The last two terms can be estimated as 
\begin{equation}\label{eq:est_K8K9}
|K_8|+|K_9|\leq \epsilon\esp\left(\int_{Q}\eta\lambda^3\phi_m^3\theta^2|v_x|^2\dx\dt\right)+C\esp\left(\int_{Q}\eta\lambda^3\phi_m^3\theta^2(|u|^2+|u_x|^2)\dx\dt\right),
\end{equation}
where the constant $C$ depends on $\|d_i\|_{L^\infty_{\mathcal F}(0,T;W^{2,\infty}(\dom))}$, $i=1,2$ and $\|d_3\|_{L^\infty_{\mathcal F}(0,T;\mathbb R)}$.

Summarizing, we collect estimates \eqref{eq:est_K1} and \eqref{eq:est_K2}--\eqref{eq:est_K8K9}, use the properties for $\eta$ and employ them on identity \eqref{eq:iden_zeta_v} to obtain
\begin{align*}\notag
\esp\left(\int_{Q_{\dom_4}}\lambda^3\phi_m^3\theta^2|v_x|^2\dx\dt\right) &\leq 5\delta\esp\left(\int_{Q}\lambda\phi_m\theta^2|v_{xx}|^2\dx\dt\right)+6\epsilon\esp\left(\int_{Q}\lambda^3\phi_m^3\theta^2|v_{x}|^2\dx\dt\right) \\ \notag
&\quad +C_{\delta,\epsilon}\esp\left(\int_{Q_{\dom_3}}\theta^2\lambda^7\phi_m^7|u|^2\dx\dt+\int_{Q_{\dom_3}}\theta^2\lambda^5\phi_m^5|u_x|^2\dx\dt\right) \\
&\quad +C_{\delta,\epsilon}\esp\left(\int_{Q_{\dom_3}}\theta^2\lambda^3\phi_m^3|u_{xx}|^2\dx\dt+\int_{Q_{\dom_3}}\theta^2\lambda^5\phi_m^5|u_{xxx}|^2\dx\dt\right).
\end{align*}
Then, using the above inequality in \eqref{eq:car_locales_u_v} and taking $\epsilon$ and $\delta$ small enough, we get
\begin{equation}\label{eq:est_car_locals_u}
\begin{split}
I_H(v_x)+I_{KS}(u)&\leq C\esp\left(\int_{Q_{\dom_3}}\theta^2\lambda^7\phi_m^7|u|^2\dx\dt+\int_{Q_{\dom_3}}\theta^2\lambda^5\phi_m^5|u_x|^2\dx\dt\right) \\
&\quad +C\esp\left(\int_{Q_{\dom_3}}\theta^2\lambda^3\phi_m^3|u_{xx}|^2\dx\dt+\int_{Q_{\dom_3}}\theta^2\lambda^5\phi_m^5|u_{xxx}|^2\dx\dt\right).
\end{split}
\end{equation}
\subsubsection*{Step 4: Local energy estimates for $u$ and their derivatives}

In the previous step, we have estimated $v_x$  in terms of local integrals of $u$ and its derivatives, so the variable $v$ does not longer appear on the right-hand side. 

At this point, estimate \eqref{eq:est_car_locals_u} looks quite similar to its deterministic counterpart (cf. \cite[Proof of Theorem 3.1]{cmp2}). However, unlike that case, we cannot estimate the local term of $u_{xxx}$ just by integrating by parts since we do not have a global term of $u_{xxxx}$ on the left-hand side of \eqref{eq:est_car_locals_u}. 

Instead, we have the following result.

\begin{lem}\label{lem:local_u3x}
Consider an open set $\dom_{2}$ such that $\dom_3\subset\subset\dom_{2}\subset\subset\dom_0$. Then, there exists $C>0$ such that
\begin{align}\notag 
\esp&\left(\int_{Q_{\dom_3}}\lambda^5\phi_m^5\theta^2|u_{xxx}|^2\dx\dt\right)\\ \notag
&\leq 4\epsilon\esp\left(\int_{Q}\lambda\phi_m\theta^2|u_{xxx}|^2\dx\dt\right)+2\delta\esp\left(\int_{Q}\lambda\phi_m\theta^2|v_{xx}|^2\dx\dt\right) \\ \notag
&+\rho\esp\left(\int_{Q}\lambda^3\phi_m^3\theta^2|v_{x}|^2\dx\dt\right)+C\esp\left(\int_{Q_{\dom_{2}}}\lambda^{13}\phi_m^{13}\theta^2|u_x|^2\dx\dt\right) \\ \label{eq:lemma_uxxx}
&+C\esp\left(\int_{Q_{\dom_{2}}}\lambda^{15}\phi_m^{15}\theta^2|u|^2\dx\dt\right) +C\esp\left(\int_{Q_{\dom_{2}}}\lambda^{7}\phi_m^{7}\theta^2|u_{xx}|^2\dx\dt\right)
\end{align}
for any $\epsilon,\delta,\rho>0$. 
\end{lem}
The idea of the proof is to obtain the differential of a suitable product and argue as in the previous step. To avoid too much word repetition, we present a brief proof on \Cref{proof_lemma_u3x}.

Using \eqref{eq:lemma_uxxx} in \eqref{eq:est_car_locals_u} and taking $\epsilon, \delta$ and $\rho$ small enough we get 
\begin{equation}\label{eq:car_locals_um3x}
\begin{split}
I_H(v_x)+I_{KS}(u)&\leq C\esp\left(\int_{Q_{\dom_{2}}}\theta^2\lambda^{15}\phi_m^{15}|u|^2\dx\dt+\int_{Q_{\dom_{2}}}\theta^2\lambda^{13}\phi_m^{13}|u_x|^2\dx\dt\right) \\
&\quad +C\esp\left(\int_{Q_{\dom_2{}}}\theta^2\lambda^7\phi_m^7|u_{xx}|^2\dx\dt\right).
\end{split}
\end{equation}

Now, taking $\dom_{1}$ with $\dom_{2}\subset\subset\dom_{1}\subset\subset\dom_0$ and constructing a cut-off function $\eta_2\in C_0^\infty(\dom_{1})$ such that $\eta_2\equiv 1$ in $\dom_{2}$, we estimate
\begin{align}\notag 
\esp\left(\int_{Q_{\dom_2}}\theta^2\lambda^7\phi_m^7|u_{xx}|^2\dx\dt\right) &\leq \esp\left(\int_{Q}\eta_2\theta^2\lambda^7\phi_m^7|u_{xx}|^2\dx\dt\right) \\ \notag
&=-\esp\left(\int_{Q}\eta_2\theta^2\lambda^7\phi_m^7u_{xxx}u_x\dx\dt\right)\\ \notag 
&\quad +\frac{1}{2}\esp\left(\int_{Q}(\eta_2\lambda^7\phi_m^7\theta^2)_{xx} |u_x|^2 \dx\dt \right) \\ \notag
&\leq \epsilon \esp\left(\int_{Q}\theta^2\lambda\phi_m|u_{xxx}|^2\dx\dt\right)\\ \label{eq:est_u2x_local}
&\quad +C_\epsilon \esp \left(\iint_{Q_{\dom_{1}}}\theta^2\lambda^{13}\phi_m^{13}|u_x|^2\dx\dt\right).
\end{align}
Moreover, taking $\eta_3\in C_0^\infty(\dom_0)$ such that $\eta_3\equiv 1$ in $\dom_{1}$, we can argue in the same way to obtain
\begin{align}\notag 
\esp\left(\int_{Q_{\dom_{1}}}\theta^2\lambda^{13}\phi_m^{13}|u_x|^2\dx\dt\right) &\leq \esp\left(\int_{Q}\eta_3\lambda^{13}\phi_m^{13}|u_x|^2\dx\dt\right) \\ \notag
&\leq \epsilon \esp\left(\int_{Q}\theta^2\lambda^3\phi_m^3|u_{xx}|^2\dx\dt\right)\\ \label{eq:est_local_final}
&\quad +C_\epsilon \esp\left(\int_{Q_{\dom_0}}\theta^2\lambda^{23}\phi_m^{23}|u|^2\dx\dt\right).
\end{align}
Putting together \eqref{eq:car_locals_um3x}, \eqref{eq:est_u2x_local} and \eqref{eq:est_local_final} and taking $\epsilon>0$ sufficiently small, we obtain the desired result. This ends the proof. 
\end{proof}

\subsection{The observability inequality}
Once we have obtained the Carleman estimate \eqref{eq:car_final}, the observability inequality \eqref{eq:obs_forward} follows immediately.

\begin{proof}[Proof of \Cref{prop:observability_forward}]
The proof is classical and follows well-known arguments (see, e.g., \cite{FCG06} in the deterministic setting). For completeness, we sketch it briefly. 

Using the properties of the weight functions $\theta$ and $\phi_m$, it is not difficult to see that 
\begin{align*}
(\lambda\phi_m)^{23}\theta^2\leq C, \quad &\forall (x,t)\in Q, \\
(\lambda\phi_m)^j \theta^2 \geq C_j, \quad &\forall (x,t)\in \dom\times\left(\tfrac{T}{4},\tfrac{3T}{4}\right), \;\; j=3,7,
\end{align*}
where the constant $C>0$ only depends on $\dom$, $\dom_0$, $m$ and $T$. Therefore, we get from \eqref{eq:car_final}
\begin{equation}\label{eq:clean_car}
\esp\left(\intdoble{\frac{T}{4}}{\frac{3T}{4}}{\dom}\left(|v_x|^2+|u|^2\right)\dx\dt\right)\leq C\esp\left(\iint_{Q_{\dom_0}}|u|^2\dx\dt\right).
\end{equation}
Using It\^{o}'s formula, we compute $\d v^2=2v\d v+(\d v)^2$ and using the equation verified by $v$, we deduce
\begin{align*}
\esp\left(\int_{\dom}|v(t_2)|^2\dx\right)-\esp\left(\int_{\dom}|v(t_1)|^2\dx\right)=&\ 2\esp\left(\intdoble{t_1}{t_2}{\dom}\left(\Gamma v_{xx}+u_x+v_x\right)v\dx\dt\right) \\
&+\esp\left(\intdoble{t_1}{t_2}{\dom}|d_2 u+d_3v|^2\dx\dt\right)
\end{align*}
for all $0\leq t_1<t_2\leq T$. Integrating by parts in the space variable, we get
\begin{align}\notag 
\esp&\left(\int_{\dom}|v(t_2)|^2\dx\right)+2\Gamma\esp\left(\intdoble{t_1}{t_2}{\dom}|v_x|^2\dx\dt\right) \\ \notag
&\quad =\esp\left(\int_{\dom}|v(t_1)|^2\dx\right)-2\esp\left(\intdoble{t_1}{t_2}{\dom}\left (u -v \right)v_x\dx\dt\right) \\ \label{eq:ener_v}
&\quad\quad  +\esp\left(\intdoble{t_1}{t_2}{\dom}|d_2 u+d_3v|^2\dx\dt\right).
\end{align}
Arguing in the same way for the variable $u$, we may obtain
\begin{align}\notag
\esp&\left(\int_{\dom}|u(t_2)|^2\dx\right)+2\gamma\esp\left(\intdoble{t_1}{t_2}{\dom}|u_{xx}|^2\dx\dt\right) \\ \notag 
&\quad = \esp\left(\int_{\dom}|u(t_1)|^2\dx\right)+2\esp\left(\intdoble{t_1}{t_2}{\dom}u_{x}u_{xx}\dx\dt \right)+2\esp\left(\intdoble{t_1}{t_2}{\dom}|u_x|^2\dx\dt\right) \\ \label{eq:ener_u}
&\quad \quad  + 2\esp\left(\intdoble{t_1}{t_2}{\dom}uv_{x}\dx\dt \right)+\esp\left(\intdoble{t_1}{t_2}{\dom}|d_1 u|^2\dx\dt\right).
\end{align}
Combining estimates \eqref{eq:ener_v}--\eqref{eq:ener_u} and using Cauchy-Schwarz and Young inequalities we get
\begin{align}\notag 
&\esp\left(\int_{\dom}\left(|u(t_2)|^2+|v(t_2)|^2\right)\dx\right)+\frac{\gamma}{2}\esp\left(\intdoble{t_1}{t_2}{\dom}|u_{xx}|^2\dx\dt\right)+\frac{\Gamma}{2}\esp\left(\intdoble{t_1}{t_2}{\dom}|v_x|^2\dx\dt\right) \\ \label{eq:est_energy}
&\leq C\esp\left(\intdoble{t_1}{t_2}{\dom}\left(|u|^2+|v|^2\right)\dx\dt\right)+C\esp\left(\int_{\dom}\left(|u(t_1)|^2+|v(t_1)|^2\right)\dx\right),
\end{align}
where $C>0$ depends on $\gamma$, $\Gamma$ and the norms of $d_i$, $i=1,2,3$. Here, we also used the inequality 
\begin{equation*}
\int_{\dom}|u_x|^2\dx\leq \epsilon \int_{\dom}|u_{xx}|^2\dx+C_\epsilon \int_{\dom} |u|^2\dx, \quad\text{for all }\epsilon>0.
\end{equation*}
Using Gronwall inequality and then integrating from $(\frac{T}{4},\frac{3T}{4})$, we get from \eqref{eq:est_energy}
\begin{equation}\label{eq:est_ener_final}
\frac{T}{2}\esp \left(\int_{\dom}\left(|v(T)|^2+|u(T)|^2\right)\dx\right)\leq C\esp\left(\intdoble{\frac{T}{4}}{\frac{3T}{4}}{\dom}\left(|v|^2+|u|^2\right)\dx\dt\right).
\end{equation}
The result follows by combining \eqref{eq:clean_car} with \eqref{eq:est_ener_final} and employing Poincar\'e inequality. 
\end{proof}

\subsection{Null controllability result}\label{sec:control_proof}
The proof is standard and follows well-known arguments, see, for instance, \cite{TZ09, LL12}.
\begin{proof}[Proof of Theorem \ref{main_teo}]
We introduce the linear subspace of $L^2_{\mathcal F}(0,T;L^2(\dom_0))$
\begin{equation*}
\mathcal X=\left\{u|_{Q_{\dom_0} \times \Omega}\; | \; (u,v) \textnormal{ solve } \eqref{eq:adjoint_KSS} \textnormal{ with some } (u_0,v_0)\in L^2(\Omega,\mathcal F_0;L^2(\dom)^2)\right\}
\end{equation*}
and define the linear functional on $\mathcal X$ as 
\begin{equation*}
\mathcal L(u|_{Q_{\dom_0}\times \Omega}):=\esp\left(\int_{\dom} (u(T)y_T+v(T)z_T)\dx\right).
\end{equation*}
Note that $\mathcal L$ is a bounded linear functional on $\mathcal X$. Indeed, by means of Cauchy-Schwarz inequality and Proposition \ref{prop:observability_forward}, we have
\begin{equation*}
|\mathcal L(u|_{Q_{\dom_0}\times \Omega})| \leq \sqrt{C}\left(\esp \int_{Q_{\dom_0}}|u|^2\dx\dt\right)^{1/2}\left(\esp\int_{\dom}(|y_T|^2+|z_T|^2)\dx\right)^{1/2}.
\end{equation*}
where $C$ is the constant appearing in \eqref{eq:obs_forward}. Using Hahn-Banach theorem, $\mathcal L$ can be extended to a bounded linear function of $L^2_{\mathcal F}(0,T;L^2(\dom_0))$ and, for the sake of simplicity, we use the same notation for the extension. From Riesz representation theorem we can find a random field $h\in L^2_{\mathcal F}(0,T;L^2(\dom_0))$ such that 
\begin{equation}\label{eq:riesz}
\esp\left(\int_{\dom} \left(u(T)y_T+v(T)z_T\right)\dx\right)=\esp\left(\int_{Q_{\dom_0}}h u\dx\dt \right).
\end{equation}

We claim that this $h$ is exactly the control that drives the solution of $(y,z)$ to zero. Using It\^{o}'s formula, we compute both $\d(y u)$ and $\d(z v)$ and after integration by parts, we get
\begin{equation*}
\esp\left(\int_{\dom}y_T u(T)\dx\right)-\esp\left(\int_{\dom}y(0) u_0\dx\right)=\esp\left(\int_{Q} (z_x -d_2 Z + h\chi_{\dom_0})u\dx\dt\right)-\esp\left(\int_{Q}v y_x\dx\dt\right)
\end{equation*}
and 
\begin{equation*}
\esp\left(\int_{\dom}z_T v(T)\dx\right)-\esp\left(\int_{\dom}z(0) v_0\dx\right)=\esp\left(\int_{Q} y_xv\dx\dt\right)+\esp\left(\int_{Q}(d_2Z-z_x)u\dx\dt\right).
\end{equation*}

Adding the previous expressions, we obtain
\begin{equation}\label{eq:duality}
\esp\left(\int_{\dom}(y_T u(T)+z_Tv(T))\dx\right)-\esp\left(\int_{\dom}(y(0) u_0+z(0)v_0)\dx\right)=\esp\left(\int_{Q_{\dom_0}}h u\dx\dt\right)
\end{equation}
and comparing \eqref{eq:duality} and \eqref{eq:riesz} yields
\begin{equation*}
\esp\left(\int_{\dom}(y(0) u_0+z(0)v_0)\dx\right)=0.
\end{equation*}
Since $(u_0,v_0)$ can be chosen arbitrarily in $L^2(\Omega,\mathcal F_0;L^2(\dom)^2)$, we have $(y(0), z(0))=0$ in $\dom$, $P$-a.s. This ends the proof.
\end{proof}

\section{Null controllability of the forward system}\label{sec:forward_cont}

In this section, we prove \Cref{th:mainresult1,th:mainresult2}. The proofs are based on the methodology proposed in \cite{HSLBP20a}. For clarity, we have divided them in several steps, organized as follows

\begin{itemize}
\item Step 1: the control cost for the linear system. Here, we will prove a controllability result for system \eqref{forward_KS-H_nonlin} with $a\equiv 0$ with an explicit control cost. We emphasize that unlike \cite[Theorem 1.1]{Liu14b}, here we provide an explicit constant which is used in the following step.

\item Step 2: source term method. In this step, we use the classical methodology of \cite{LLT13}, recently adapted in the stochastic setting in \cite{HSLBP20a}, to prove a controllability result for \eqref{forward_KS-H_nonlin} (with $a\equiv 0$) with right-hand side terms vanishing as $t\to T$. At this point, the control cost obtained before is fundamental. 

\item Step 3: regular controlled trajectories. Using the result coming from the previous step, we employ some regularity estimates to provide a more regular controlled trajectory. This will allow us to analyze the nonlinear system \eqref{forward_KS-H_nonlin} and its controllability properties. 

\item Step 4: global and \emph{statistical} null controllability. Finally, in this point, using a fixed point argument we prove \Cref{th:mainresult1} for any initial data $(y_0,z_0)$ and a small enough truncation constant $R$. Then, imposing smallness conditions on the size of the initial datum, Markov's inequality allows us to achieve the result presented in \Cref{th:mainresult2}. 
\end{itemize}

\subsection{Step 1: the control cost for the linear system}

Let us fix $T>0$ and consider the following linear system

\begin{equation}\label{forward_KS-H_linear}
\begin{cases}
\d y+y_{xxxx}\,\dt=\chi_{\dom_0}h\,\dt+(b_1y+ b_2 z)\, \d W(t) &\text{in }Q, \\
\d z- z_{xx}\,\dt=y\,\dt+b_3z\, \d W(t) &\text{in }Q, \\
y=y_{xx}=0 &\text{on }\Sigma, \\
z=0 &\text{on }\Sigma, \\
y(0)=y_0, \quad z(0)=z_0 &\text{in } \dom. 
\end{cases}
\end{equation}
This system arises after linearizing \eqref{forward_KS-H_nonlin} around (0,0). According to \Cref{prop:app_well_2}, point a), for given $(y_0,z_0)\in L^2(\Omega,\mathcal F_0; L^2(\dom)^2)$ and $h\in L^2_{\fil}(0,T;L^2(\dom_0))$, \eqref{forward_KS-H_linear} has a unique solution $(y,z)\in L^2_{\mathcal F}(0,T;H^2(\dom)\times H_0^{1}(\dom)) \bigcap L^2\left(\Omega; C([0,T];L^2(\dom)^2)\right)$.

We present the following result.

\begin{prop}\label{prop:with_cost_control}
There exists a positive constant $C=C(\dom,\dom_0,b_1,b_2,b_3)$ such that for every $T>0$, $(y_0,z_0)\in L^2(\Omega,\mathcal F_0; L^2(\dom)^2)$, there exists $h\in \in L^2_{\fil}(0,T;L^2(\dom_0))$ such that the solution $(y,z)$ to \eqref{forward_KS-H_linear} satisfies
\[
y(T)=z(T)=0 \quad\text{in } \dom, \ a.s.
\]
Moreover, the control $h$ satisfies the following estimate 
\begin{equation*}
\esp\left(\int_{Q_{\dom_0}}|h|^2\dx\d{t}\right)\leq C_{T}\esp\left(\|y_0\|^2_{L^2(\dom)}+\|z_0\|^2_{L^2(\dom)}\right)
\end{equation*}
with 
\begin{equation*}
C_T=e^{C/T}.
\end{equation*}
\end{prop}

The proof of this result is based on the well-known Lebeau-Robbiano method. For its development, we shall present first a series of auxiliary lemmas. 

We consider the linear operator $\mathcal A$ in $L^2(\dom)$
\begin{equation}\label{eq:def_op_A}
\mathcal Az:=-z_{xx}, \quad\forall z\in H^2(\dom)\cap H^1_0(\dom).
\end{equation}
Let $\{\lambda_i\}_{i=1}^{\infty}$ be the eigenvalues of $\mathcal A$ and $\{\varphi_i\}_{i=1}^{\infty}$ be the corresponding (normalized) eigenfunctions. The family $\{\varphi_i\}_{i=1}^{\infty}$ is an orthonormal basis of $L^2(\dom)$. Also, in this simple case, one can see that $\lambda_i=(i\pi)^2$ and $\varphi_i=\sqrt{2}\sin(i\pi x)$.

According to \cite{LRR19}, if $\mathcal A^2u:=u_{xxxx}$ on $\dom$ together with the boundary conditions $u=u_{xx}=0$ on $\{0,1\}$, the family $\{\varphi_i\}_{i=1}^{\infty}$ is actually composed of eigenfunctions of $\mathcal A^2$ associated to the eigenvalues $\mu_i=\lambda_i^2$. Moreover, this operator satisfies the following spectral inequality.
\begin{lem}\label{lem:spec_4}
Let $\dom_0$ be an open subset of $\dom$. There exists $C>0$ such that for any  $r>0$ it holds
\begin{equation*}
\|z\|_{L^2(\dom)}\leq Ce^{C r ^{1/4}}\|z\|_{L^2(\dom_0)}
\end{equation*}
for every $z\in\textnormal{span}\{\varphi_i\}_{\mu_i\leq r}$.
\end{lem}

\begin{rmk} We are using here the boundary conditions $u=u_{xx}=0$ on $\{0,1\}$ to exploit the fact that both operators share eigenfunctions and that $\mu_i=\lambda_i^2$, which is not the case for the fourth-order operator with clamped boundary conditions, i.e., $u=u_{x}=0$ on $\{0,1\}$. We refer the reader to \cite{AE13,Gao_lebeau,LRR19} for variants of Lemma \ref{lem:spec_4} corresponding to the clamped boundary conditions.
\end{rmk}

Now, for any $\tau>0$, we set $Q_\tau=\dom \times (0,\tau)$, $\Sigma_\tau=\dom\times(0,\tau)$ and consider the following backward stochastic system
\begin{equation}\label{eq:system_back_tau}
\begin{cases}
\d u-u_{xxxx}\dt=(-v-b_1\ov{u})\dt+\ov{u}\,\d W(t) &\text{in } Q_\tau, \\
\d v+v_{xx}\dt=(-b_2\ov{u}-b_3\ov{v})\dt+\ov{v}\,\d W(t) &\text{in } Q_\tau, \\
u=u_{xx}=0 &\text{on }\Sigma_\tau, \\
v=0 &\text{on } \Sigma_\tau, \\
u(\tau)=u_\tau, \quad v(\tau)=v_\tau &\text{in } \dom.
\end{cases}
\end{equation}
For any terminal data $(u_{\tau},v_\tau)\in L^2(\Omega,\mathcal F_{\tau}; L^2(\dom)^2)$, we know thanks to \Cref{prop:back_uxx} that system \eqref{eq:system_back_tau} has a unique solution
\begin{equation*}
(u,v,\ov{u},\ov{v})\in \left[L^2_{\mathcal F}(0,\tau;H^2(\dom)\times H_0^{1}(\dom)) \bigcap L^2\left(\Omega; C([0,\tau];L^2(\dom)^2)\right)\right]\times L^2_{\mathcal F}(0,\tau;L^2(\dom)^2).
\end{equation*}

For each $r>0$, consider the space $\mathcal X_r=\textnormal{span}\{\varphi_i\}_{\lambda_i\leq \sqrt{r}}$ and denote by $\Pi_r$ the orthogonal projection from $L^2(\dom)$ to $\mathcal X_r$. {We present the following result, which is an observability inequality with only one observation for system \eqref{eq:system_back_tau} with final data in $\mathcal X_r$.}

\begin{lem}\label{prop:obs_ineq_Xk}
There exists a positive constant {$C=C(\dom,\dom_0,b_1,b_2,b_3)$ such that for any $\tau\in(0,1)$, any $\sqrt{r}\geq \lambda_1$}, and $(u_{\tau},v_{\tau})\in L^2(\Omega,\mathcal{F}_\tau;(\mathcal X_r)^2)$, the corresponding solution $(u,v, \overline u,\overline v)$ of \eqref{eq:system_back_tau} satisfies:

{
\begin{equation}\label{eq:obs_particular}
\esp \left(\|u(0)\|_{L^2(\dom)}^2\right)+\esp\left(\|v(0)\|^2_{L^2(\dom)}\right)\leq C\frac{e^{C{r}^{1/4}}}{\tau}\esp\left(\int_{0}^{\tau}\!\!\! \int_{\dom_0}|u|^2\dx\dt\right).
\end{equation}
}
\end{lem}

{ This result is a particular case of \cite[Proposition 3.1]{Liu14b}. Indeed, in the way we are defining the operator $\mathcal A$ in \eqref{eq:def_op_A} it is enough to take $\alpha=2$ in the work \cite{Liu14b} to define the fourth-order operator. A careful inspection of the proof of \cite[Proposition 3.1]{Liu14b} and \Cref{lem:spec_4} allows us to conclude the result without major changes. 
}

With the observability inequality in \Cref{prop:obs_ineq_Xk}, one can prove the following controllability result.

\begin{lem}\label{prop:size_control}
For each $\sqrt{r}\geq \lambda_1$ and any $\tau\in(0,1)$, there exists $h_r\in L^2_{\mathcal{F}}(0,\tau;L^2(\dom_0))$ such that the corresponding controlled solution $(y,z)$ to \eqref{forward_KS-H_linear} satisfies
\begin{equation*}
\Pi_r(y(\tau))=\Pi_r(z(\tau))=0 \quad \text{in } \dom, \ P\textnormal{-a.s.}
\end{equation*}
Moreover, we can estimate the control cost and the size of the controlled solution as
\begin{align}\label{eq:est_control_case1}
\|h_r\|^2_{L^2_{\mathcal F}(0,\tau;L^2(\dom_0))}&\leq C\frac{e^{C{r}^{1/4}}}{\tau}\esp \left(\|y_0\|^2_{L^2(\dom)}+\|z_0\|^2_{L^2(\dom)}\right)
\end{align}
and
\begin{align}\label{eq:est_size_case1}
\esp\left(\|y(\tau)\|_{L^2(\Omega)}^2+\|z(\tau)\|^2_{L^2(\Omega)}\right) &\leq C\frac{e^{C{r}^{1/4}}}{\tau}\esp \left(\|y_0\|^2_{L^2(\dom)}+\|z_0\|^2_{L^2(\dom)}\right),
\end{align}
where $C=C(\dom,\dom_0,b_1,b_2,b_3)>0$ uniform with respect to $\tau$. 
\end{lem}

The arguments needed to prove this lemma are similar to those presented in the proof of Theorem \ref{main_teo} and only minor adaptions are required, see Section \ref{sec:control_proof}.

We also need a dissipation result for the uncontrolled solution. We present the following. 
\begin{prop}\label{prop:dissipation}
Assume that $h\equiv 0$ in system \eqref{forward_KS-H_linear}. For any $(y_0,z_0)\in L^2_{\mathcal F}(\Omega,\mathcal F_0; L^2(\dom)^2)$ with $\Pi_{\lambda_k}(y_0)=\Pi_{\lambda_k}(z_0)=0$ in $\dom$, $P$-a.s., the corresponding solution $(y,z)$ satisfies
\begin{equation*}
\esp\left(\|y(t)\|^2_{L^2(\dom)}+\|z(t)\|^2_{L^2(\dom)}\right)\leq e^{-\gamma_{k+1} t}\esp\left(\|y_0\|^{2}_{L^2(\dom)}+\|z_0\|^{2}_{L^2(\dom)}\right), \quad \forall t\in[0,T],
\end{equation*}
where $\gamma_{k+1}=2\lambda_{k+1}-\sigma$ where $\sigma=1+4\left(\sum_{i=1}^{3}\|b_i\|^2_{L^\infty_{\mathcal F}(0,T;\R)}\right)$.
\end{prop}
The proof of this proposition is standard and can be done by following almost the same procedure as in \cite[Proposition 2.3]{LU11} and \cite[Proposition 4.1]{Liu14b}.

Using the above results, we prove the following corollary, which will be of interest during the proof of \Cref{prop:with_cost_control}.

\begin{cor}\label{cor:corrolary_franck}
For each $\sqrt{r}\geq \lambda_1$, $\tau\in(0,1)$ and $(y_0,z_0)\in L^2(\Omega,\mathcal F_{0};L^2(\dom)^2)$, there exists a control $h_r\in L^2_{\mathcal F}(0,\tau;L^2(\dom_0))$ such that
\begin{align*}
\|h_r\|^2_{L^2_{\mathcal F}(0,\tau;L^2(\dom_0))}&\leq C\frac{e^{C{r}^{1/4}}}{\tau}\esp \left(\|y_0\|^2_{L^2(\dom)}+\|z_0\|^2_{L^2(\dom)}\right)
\end{align*}
and
\begin{align*}
\displaystyle \esp\left(\|y(\tau)\|_{L^2(\dom)}^2+\|z(\tau)\|^2_{L^2(\dom)}\right) &\leq \frac{C}{\tau}{e^{C{r}^{1/4}-{\sqrt{r}\tau }}}\esp \left(\|y_0\|^2_{L^2(\dom)}+\|z_0\|^2_{L^2(\dom)}\right).
\end{align*}
\end{cor}
\begin{proof}
The proof is straightforward. Use \Cref{prop:size_control} in the interval $(0,\frac{\tau}{2})$, this will give a control $w_r$ such that $\Pi_r(y(\tau/2))=\Pi_r(z(\tau/2))=0$ together with the estimates \eqref{eq:est_control_case1} and \eqref{eq:est_size_case1}, but with $\tau$ replaced by $\tau/2$. Set
\begin{equation*}
h_r=
\begin{cases}
w_r & \text{for } t\in(0,\tau/2),  \\
0 & \text{for } t\in(\tau/2,\tau).
\end{cases}
\end{equation*}
Clearly, $h_r$ and $w_r$ have the same norm. 

Now, note that there exists $k\in\mathbb N^*$ such that $\lambda_{k+1}>\sqrt{r}$. Applying \Cref{prop:dissipation} in the interval $(\tau/2,\tau)$ and the fact that $\tau\in(0,1)$ allow us to deduce  
\begin{align}\notag
\esp\left(\|y(\tau)\|^2_{L^2(\dom)}+\|z(\tau)\|^2_{L^2(\dom)}\right) & \leq e^{-(2\lambda_{k+1}-\sigma) \frac{\tau}{2}}\esp\left(\|y(\tau/2)\|^{2}_{L^2(\dom)}+\|z(\tau/2)\|^{2}_{L^2(\dom)}\right) \\ \label{est_mm}
&\leq C e^{-\lambda_{k+1}\tau} \esp\left(\|y(\tau/2)\|^{2}_{L^2(\dom)}+\|z(\tau/2)\|^{2}_{L^2(\dom)}\right),
\end{align}
since $\Pi_r(y(\tau/2))=\Pi_r(z(\tau/2))=0$ implies that the first $k$ modes of the equations have been killed. Using that $\lambda_{k+1}>\sqrt{r}$ in the above estimate and combining the result  with \eqref{eq:est_size_case1} (replacing $\tau$ replaced by $\tau/2$), we obtain the desired result. Thus, the proof is complete. 
\end{proof}

Now, we are in position to present the proof of \Cref{prop:with_cost_control}. 

\begin{proof}[Proof of \Cref{prop:with_cost_control}]
We follow the spirit of \cite[Section IV.2]{franck_calavera} and \cite{LU11}. We will pay special attention to the dependence with respect to $T$ in the following estimates. In what follows, $C$ stands for a positive constant changing from line to line depending at most on $\dom$, $\dom_0$, and the coefficients $b_i$.

The idea is to split the time interval $(0,T)$ into subintervals of size $\tau_j$, $j\geq 1$, with
\begin{equation*}
\sum_{j=1}^{\infty}\tau_j=T
\end{equation*}
and apply successively a partial control as in \Cref{cor:corrolary_franck} with a cut frequency $r_j$ tending to infinity as $j\to \infty$. We set
\begin{equation}\label{eq:def_taus_lambdas}
\tau_j=\frac{T}{2^j} \quad\text{and}\quad r_j=\beta^2(2^j)^4
\end{equation}
for some $\beta>0$ to be determined. 

Let $T_j=\sum_{k=1}^{j}\tau_k$, for $j\geq 1$. We proceed as follows.
\begin{enumerate}
\item During the interval $(0,\tau_1)=(0,T_1)$, we apply a control $h_{r_1}$ as given in \Cref{cor:corrolary_franck} with $r={r}_1$, in such a way that
\begin{align*}
\|h_{r_1}\|^2_{L^2_{\mathcal F}(0,T_1;L^2(\dom_0))}&\leq C\frac{e^{C{r_1^{1/4}}}}{\tau_1}\esp \left(\|y_0\|^2_{L^2(\dom)}+\|z_0\|^2_{L^2(\dom)}\right)
\end{align*}
and
\begin{align*}
\esp\left(\|y(T_1)\|_{L^2(\dom)}^2+\|z(T_1)\|^2_{L^2(\dom)}\right) &\leq \frac{C}{\tau_1}{e^{C{r_1^{1/4}}-{\sqrt{r_1}\tau_1 }}}\esp \left(\|y_0\|^2_{L^2(\dom)}+\|z_0\|^2_{L^2(\dom)}\right).
\end{align*}
with 
$$\Pi_{r_1}(y(T_1))=0=\Pi_{r_1}(z(T_1))=0, \quad P\textnormal{-a.s.}$$
\item During the interval $(\tau_1,\tau_1+\tau_2)$, we apply a control $h_{r_2}$ once again given by \Cref{cor:corrolary_franck} with $r=r_2$ in such a way that
\begin{align*}
\|h_{r_2}\|^2_{L^2_{\mathcal F}(T_1,T_2;L^2(\dom_0))}&\leq C\frac{e^{C{r_2^{1/4}}}}{\tau_2}\esp \left(\|y(T_1)\|^2_{L^2(\dom)}+\|z(T_1)\|^2_{L^2(\dom)}\right)
\end{align*}
and
\begin{align*}
\esp\left(\|y(T_2)\|_{L^2(\dom)}^2+\|z(T_2)\|^2_{L^2(\dom)}\right) &\leq \frac{C^2}{\tau_1\tau_2}{e^{C({r_1^{1/4}}+r_2^{1/4})-{\sqrt{r_1}\tau_1 }-{\sqrt{r_2}\tau_2}}}\esp \left(\|y_0\|^2_{L^2(\dom)}+\|z_0\|^2_{L^2(\dom)}\right).
\end{align*}
with
$$\Pi_{r_2}(y(T_2)=0)=\Pi_{r_2}(z(T_2))=0, \quad P\textnormal{-a.s.}$$
\item By an inductive procedure, we can build a control $h_{r_j}$ on the time interval $(T_{j-1},T_j)$ such that
\begin{align*}
\|h_{r_j}\|^2_{L^2_{\mathcal F}(T_{j-1},T_j;L^2(\dom_0))}&\leq C\frac{e^{C{r}_j^{1/4}}}{\tau_j}\esp \left(\|y(T_{j-1})\|^2_{L^2(\dom)}+\|z(T_{j-1})\|^2_{L^2(\dom)}\right)
\end{align*}
and
\begin{align*}
\esp&\left(\|y(T_j)\|_{L^2(\dom)}^2+\|z(T_j)\|^2_{L^2(\dom)}\right) \\
&\quad \leq \frac{C^j}{\prod_{k=1}^{j}\tau_k}{e^{C\left(\sum_{k=1}^{j}{r_k^{1/4}}\right)-\sum_{k=1}^{j}\tau_k\sqrt{r_k} }}\esp \left(\|y_0\|^2_{L^2(\dom)}+\|z_0\|^2_{L^2(\dom)}\right).
\end{align*}
with
\begin{equation}\label{eq:proy_tj}
\Pi_{r_j}(y(T_j)=0)=\Pi_{r_j}(z(T_j))=0, \quad P\textnormal{-a.s.}
\end{equation}

{Moreover, recalling the definitions in \eqref{eq:def_taus_lambdas}, we can refine the above estimates as follows
\begin{align}\label{eq:est_control_j}
\|h_{r_j}\|^2_{L^2_{\mathcal F}(T_{j-1},T_j;L^2(\dom_0))}&\leq C\frac{e^{C\sqrt{\beta}2^j}}{T}\esp \left(\|y(T_{j-1})\|^2_{L^2(\dom)}+\|z(T_{j-1})\|^2_{L^2(\dom)}\right)
\end{align}
and
\begin{align}\notag 
\esp&\left(\|y(T_j)\|_{L^2(\dom)}^2+\|z(T_j)\|^2_{L^2(\dom)}\right) \\ \notag
&\leq \left(\frac{C}{T}\right)^j{e^{[C\sqrt{\beta}-T\beta]\sum_{k=1}^j2^k}}\esp \left(\|y_0\|^2_{L^2(\dom)}+\|z_0\|^2_{L^2(\dom)}\right) \\ \notag
& \leq e^{C/T} e^{\sum_{k=1}^j 2^k} {e^{[C\sqrt{\beta}-T\beta]\sum_{k=1}^j2^k}}\esp \left(\|y_0\|^2_{L^2(\dom)}+\|z_0\|^2_{L^2(\dom)}\right)
\\
\label{eq:est_final_refinado}
& \leq e^{C/T} {e^{[C\sqrt{\beta}-T\beta]\sum_{k=1}^j2^k}}\esp \left(\|y_0\|^2_{L^2(\dom)}+\|z_0\|^2_{L^2(\dom)}\right)
\end{align}

}

\item {We note that 
\begin{align*}
[C\sqrt{\beta}-{\beta}T]\sum_{k=1}^{j}2^k =\left(C\sqrt{\beta}-{\beta}T\right)\left(2^{j+1}-2\right)
\end{align*}
and choosing $\beta$ sufficiently large so that
$
\widetilde{\beta}:={\beta}T-C\sqrt{\beta}>0
$
we can obtain from \eqref{eq:est_final_refinado}
\begin{align} \label{eq:est_ytj}
\esp&\left(\|y(T_j)\|^2_{L^2(\dom)}+\|z(T_j)\|^2_{L^2(\dom)}\right)  \leq e^{C/T} e^{-\widetilde{\beta}2^{j+1}}\esp \left(\|y_0\|^2_{L^2(\dom)}+\|z_0\|^2_{L^2(\dom)}\right).
\end{align}
}
\item Using estimate \eqref{eq:est_ytj} in \eqref{eq:est_control_j}, we obtain
{
\begin{align*}
\|h_{r_j}\|^2_{L^2_{\mathcal F}(T_{j-1},T_{j};L^2(\dom_0))} \leq  \frac{C}{T}e^{C/T}{e^{(C\sqrt{\beta}-\widetilde{\beta})2^{j}}}\esp \left(\|y_0\|^2_{L^2(\dom)}+\|z_0\|^2_{L^2(\dom)}\right),
\end{align*}
}
and increasing the value of $\beta$ to ensure that 
$
\widehat{\beta}:=\widetilde{\beta}-C\sqrt{\beta}>0
$
we obtain
{
\begin{align}\label{eq_est_cont_final}
\|h_{r_j}\|^2_{L^2_{\mathcal F}(T_{j-1},T_j;L^2(\dom_0))} \leq  \frac{C}{T}e^{C/T}e^{-\widehat{\beta}2^j}\esp \left(\|y_0\|^2_{L^2(\dom)}+\|z_0\|^2_{L^2(\dom)}\right).
\end{align}
}
\item {Estimate \eqref{eq_est_cont_final} shows that
$
\sum_{j=1}^{\infty}\|h_{r_j}\|^2_{L^2_{\mathcal F}(T_{j-1},T_j; L^2(\dom_0))}<\infty
$
and in particular the control $h$ that comes from gluing together all the $(h_{r_j})_{j\in\mathbb N^*}$ is an element of $L^2_{\mathcal F}(0,T;L^2(\dom_0))$. From \eqref{eq:proy_tj} and \eqref{eq:est_ytj} and since $T_j\to T$ as $j\to \infty$, we conclude that
\begin{equation*}
y(T)=z(T)=0 \text{ in }\dom, \ P\text{-a.s.}
\end{equation*}

Finally, to get a precise estimate of the control cost it is enough to take $\beta=\alpha/T^2$ for some $\alpha>0$ large enough (and thus $\widetilde{\beta}$ and $\widehat{\beta}$ are proportional to $1/T$) and use \eqref{eq_est_cont_final} to obtain 
\begin{align*}
\|h\|^2_{L^2_{\mathcal F}(0,T;L^2(\dom_0))} &= \sum_{j=1}^{\infty}\|h_{r_j}\|^2_{L^2_{\mathcal F}(T_{j-1},T_j; L^2(\dom_0))} \\
& \leq \frac{C}{T}e^{C/T} \left(\int_0^{+\infty} e^{-\frac{C^\prime \sigma}{T} }\d\sigma \right) \esp \left(\|y_0\|^2_{L^2(\dom)}+\|z_0\|^2_{L^2(\dom)}\right) \\
&\leq C e^{C/T} \esp \left(\|y_0\|^2_{L^2(\dom)}+\|z_0\|^2_{L^2(\dom)}\right).
\end{align*}
}
\end{enumerate}
This concludes the proof.

\end{proof}

\subsection{Step 2: source term method}

\label{sec:stochasticsource}

Let us fix $M \geq C$ where $C=C(\dom,\dom_0,b_1,b_2,b_3)$ comes from \Cref{prop:with_cost_control} and define the weight
\begin{equation}
\label{eq:defgamma}
\forall t >0,\ \gamma(t) := e^{M/t}.
\end{equation}
Let us fix $T \in (0,1)$ and take $Q \in (1, \sqrt{2})$ and $P > Q^{2}/(2-Q^{2})$. We define the weights
\begin{align}
\label{eq:rho0}
&\forall t \in [0,T),\ \rho_0(t) :=  \exp\left(- \frac{MP}{(Q-1)(T-t)}\right),\\
\label{eq:rho}
&\forall t \in [0,T),\ \rho(t) := \exp\left(-\frac{(1+P)Q^2M}{({Q}-1)(T-t)}\right).
\end{align}
For an appropriate source term $F$, we consider
\begin{equation}\label{eq:heat_source}
\begin{cases}
\d y+(y_{xxxx}+F)\dt=\chi_{\dom_0}h\,\dt+(b_1y+ b_2 z)\, \d W(t) &\text{in }Q, \\
\d z- z_{xx}\,\dt=y\,\dt+b_3z\, \d W(t) &\text{in }Q, \\
y=y_{xx}=0 &\text{on }\Sigma, \\
z=0 &\text{on }\Sigma, \\
y(0)=y_0, \quad z(0)=z_0 &\text{in } \dom. 
\end{cases}
\end{equation}
We recall that for $F\in L^2_{\fil}(0,T;L^2(\dom))$, $h \in L^2_{\fil}(0,T;L^2(\dom_0))$ and $(y_0,z_0) \in L^2(\Omega,\fil_0;L^2(\dom)^2)$, system \eqref{eq:heat_source} is well-posed according to \Cref{prop:app_well_2}.
We define associated spaces for the source term, the state and the control
\begin{align*}
& \mathcal{S} := \left\{S \in  L^2_{\fil}(0,T;L^2(\dom)) : \frac{S}{\rho} \in  L^2_{\fil}(0,T;L^2(\dom))\right\},\\
& \mathcal{Y}:= \left\{y \in L^2_{\fil}(0,T;L^2(\dom)) : \frac{y}{\rho_0} \in L^2_{\fil}(0,T;L^2(\dom))\right\},\\
& \mathcal{H} := \left\{h \in L^2_{\fil}(0,T;L^2(\dom)) : \frac{h}{\rho_{0}}  \in L^2_{\fil}(0,T;L^2(\dom))\right\}.
\end{align*}
From the behaviors near $t=T$ of $\rho$ and $\rho_0$, we deduce that each element of $\mathcal{S}$, $\mathcal{Y}$, $\mathcal{H}$ vanishes at time $t=T$.

We have the following null-controllability result for \eqref{eq:heat_source}.
\begin{prop}\label{prop:source_stochastic}
For every $(y_0,z_0) \in L^2(\Omega,\fil_0;L^2(\dom)^2)$ and $F\in \mathcal{S}$, 
there exists a control $h \in \mathcal{H}$ such that the corresponding controlled solution $(y,z)$ to \eqref{eq:heat_source} belongs to $\mathcal{Y}\times \mathcal Y$. Moreover, there exist positive constants $C=C(\dom,\dom_0, b_1, b_2, b_3)$ and $C_T = e^{C/T}$ such that
\begin{align}\notag 
\esp & \left(\sup_{0\leq t\leq T }\left\|\frac{y(t)}{\rho_0(t)}\right\|^2_{L^2(\dom)}+\sup_{0\leq t\leq T }\left\|\frac{z(t)}{\rho_0(t)}\right\|^2_{L^2(\dom)}\right)+\esp\left(\int_{Q_{\dom_0}}\left|\frac{h}{\rho_0}\right|^2\dx\dt\right) \\ \label{eq:sup_y_limit}
&\leq C_T\esp\left(\|y_0\|^2_{L^2(\dom)}+\|z_0\|^2_{L^2(\dom)}+\int_{0}^{T}  \left\| \frac{F(t)}{\rho(t)}\right\|^2_{L^2(\dom)}\dt \right).
\end{align}
In particular, since $\rho_0$ is a continuous function satisfying $\rho_0(T)=0$, the above estimate implies 
\begin{equation*}
\label{eq:ytauzero}
y(T)=z(T)=0 \quad \textnormal{in } \dom, \ \textnormal{a.s.}
\end{equation*}
\end{prop}

The proof follows the lines of those shown in \cite[Section 2.2]{HSLBP20a}. For completeness, we give a sketch in \Cref{proof:source}.

\subsection{Step 3: regular controlled trajectories}\label{sec:regular}
\indent The next proposition gives more information on the regularity of the controlled trajectory obtained in \Cref{prop:source_stochastic}. Let us define 
\begin{equation}\label{eq:def_rho_hat_exact}
\hat{\rho}(t) =  \exp\left(- \frac{M\zeta}{(Q-1)(T-t)}\right),\quad  \text{with}\ \frac{(1+P)Q^2}{2} < \zeta < P.
\end{equation}
It is straightforward to see that
\begin{gather}
\label{eq:defrho}
\rho_0 \leq C \hat{\rho},\ \rho \leq C \hat{\rho},\ |\hat{\rho}'| \rho_0 \leq C \hat{\rho}^2, \\ \label{eq:hat_rho_bound}
 {\hat{\rho}^2 \leq C \rho}.
\end{gather}

\begin{prop}\label{prop:SourceTermReg}
For every $(y_0,z_0) \in L^2(\Omega,\fil_0;H^2(\dom)\cap H_0^1(\dom))\times L^2(\Omega,\fil_0; H_0^1(\dom)) $, $F \in \mathcal{S}$, there exists a control $h \in \mathcal{H}$, such that the solution $y$ of \eqref{eq:heat_source} satisfies the following estimate
\begin{align}\notag 
\esp & \left(\sup_{0\leq t\leq T }\left\|\frac{y(t)}{\hat\rho(t)}\right\|^2_{H^2(\dom)}+\sup_{0\leq t\leq T }\left\|\frac{z(t)}{\hat \rho(t)}\right\|^2_{H^1_0(\dom)}\right) + \esp\left(\int_{0}^{T}\left\| \frac{y(t)}{\hat\rho(t)}\right\|^2_{H^4(\dom)}\dt+\int_{0}^{T}\left\|\frac{z(t)}{\hat \rho(t)}\right\|^2_{H^2(\dom)}\dt\right)    \\ \label{eq:reg_extra}
&\leq C_T\esp\left(\|y_0\|^2_{H^2(\dom)\cap H_0^1(\dom)}+\|z_0\|^2_{H_0^1(\dom)}+\int_{0}^{T}  \left\| \frac{F(t)}{\rho(t)}\right\|^2_{L^2(\dom)} \dt \right),
\end{align}
where $C_T = e^{C/T}$ with $C=C(\dom,\dom_0, b_1,b_2, b_3)>0$.

In particular, since $\hat{\rho}$ is a continuous function satisfying $\hat{\rho}(T)=0$, the above estimate implies 
\begin{equation*}
y(T)=0 \quad \textnormal{in } \dom, \ \textnormal{a.s.}
\end{equation*}
\end{prop}
The proof of \Cref{prop:SourceTermReg} is can be obtained as in \cite[Proposition 2.6]{HSLBP20a}. It is enough to consider  a control $h \in \mathcal{H}$ and $(y,z) \in \mathcal{Y}\times \mathcal Y$ the corresponding controlled solution provided by \Cref{prop:source_stochastic} and define $w:=\frac{y}{\hat{\rho}}$, $r=\frac{z}{\hat \rho}$. By means of It\^{o}'s formula, one derives the equations in the new variable and apply the maximal regularity estimate of \Cref{prop:app_well_2} together with estimates \eqref{eq:defrho}.


\subsection{Step 4: global and \emph{statistical} null controllability}\label{sec:contr_nonlin}

\subsubsection{Proof of the global null-controllability result}

Here, we will achieve the proof \Cref{th:mainresult1}. Let us fix $T\in(0,1)$. In what follows, we denote by $C=C_T$ positive constants (that may change from line to line) of the form $e^{C/T}$ with $C>0$ depending at most on $\dom$, $\dom_0$, $b_1$, $b_2$, and $b_3$. 

For the sake of clarity, we divide the proof in three steps.

\smallskip


\textit{Step 1.} Here, we will provide a Lipschitz-type estimate for the nonlinear function $f$, where we recall that $f(y,y_x)=y y_x$. 

We take $t\in(0,T)$ and consider pairs $(y_1,z_1),(y_2,z_2)\in X_t$.  We have
\begin{align}\notag
\norme{f(y_1,y_{1,x})-f(y_2,y_{2,x})}_{L^2(\dom)}&=\norme{y_1 y_{1,x}-y_{2}y_{2,x}}_{L^2(\dom)} \\ \label{eq:norm_f_dif}
&\leq \norme{(y_1-y_2)y_{1,x}}_{L^2(\dom)}+\norme{y_2(y_{1,x}-y_{2,x})}_{L^2(\dom)}=: I_{1}+I_{2}
\end{align}
Using H\"{o}lder inequality and the definition of the $H^1$-norm, we readily see that
\begin{align*}
I_1&\leq \norme{y_1-y_2}_{L^\infty(\dom)}\norme{y_{1,x}}_{L^2(\dom)} \leq \norme{y_1-y_2}_{L^\infty(\dom)}\norme{y_1}_{H^1(\dom)}
\end{align*}
and using the Sobolev embedding gives
\begin{align}\label{eq:I1_lip}
I_1\leq C \norme{y_1-y_2}_{H^1(\dom)}\norme{y_1}_{H^1(\dom)},
\end{align}
for some constant $C>0$ only depending on $\dom$. In the same spirit, it is not difficult to see that that
\begin{align}\label{eq:I2_lip}
I_2&\leq \norme{y_2}_{L^\infty}\norme{y_{1,x}-y_{2,x}}_{L^2(\dom)} \leq C\norme{y_{2}}_{H^1(\dom)}\norme{y_1-y_2}_{H^1(\dom)}.
\end{align}
Thus, putting \eqref{eq:I1_lip}--\eqref{eq:I2_lip} in \eqref{eq:norm_f_dif} yields
\begin{equation*}
\norme{f(y_1,y_{1,x})-f(y_2,y_{2,x})}_{L^2(\dom)} \leq C \norme{y_{1}-y_{2}}_{H^1(\dom)} \left(\norme{y_1}_{H^1(\dom)}+\norme{y_2}_{H^1(\dom)}\right)
\end{equation*}

Since the weight functions $\rho$ and $\hat \rho$ are independent of the variable $x$, we can introduce them in the above inequality and using property \eqref{eq:hat_rho_bound} we can deduce
\begin{equation}\label{eq:est_lipschitz}
\norme{\frac{f(y_1,y_{1,x})-f(y_2,y_{2,x})}{\rho}} \leq C \norme{\frac{y_1-y_2}{\hat\rho}}_{H^1(\dom)} \left(\norme{\frac{y_1}{\hat \rho}}_{H^1(\dom)}+\norme{\frac{y_2}{\hat\rho}}_{H^1(\dom)}\right).
\end{equation}
This completes Step 1. 

\smallskip
\textit{Step 2.} Let us recall the functions defined in \eqref{c_tr} and \eqref{eq:def_fR}, given by
\begin{align*}
\varphi_{R}=\begin{cases}
1& \mbox{if } x\leq R, \\
0& \mbox{if } x\geq 2R, \\
\end{cases}
\end{align*}
\begin{equation*}
f_{R}(y,y_x)=\varphi_{R}(\|(y,z)\|_{X_t})f(y,y_x).
\end{equation*} 
Without loss of generality, we assume that 
\begin{equation}\label{b_0}
\|(y_2,z_2)\|_{X_t}\leq \|(y_1,z_1)\|_{X_t}.
\end{equation}
Then, applying triangle inequality we get 
\begin{align}\label{b_1}
\left\|\frac{f_{R}(y_1,y_{1,x})-f_{R}(y_2,y_{2,x})}{\rho}\right\|_{L^2(\dom)}\leq I_1+I_2,
\end{align}
where
\begin{align*}
I_1:=&\left\|\frac{\left(\varphi_{R}(\|(y_1,z_1)\|_{X_t})-\varphi_{R}(\|(y_2,z_2)\|_{X_t})\right)f(y_2,y_{2,x})}{\rho}\right\|_{L^2(\dom)},\\
I_2:=&\left\|\frac{\varphi_{R}(\|(y_1,z_1)\|_{X_t})\left(f(y_1,y_{1,x})-f(y_2,y_{2,x})\right)}{\rho}\right\|_{L^2(\dom)}.
\end{align*}
Note that the mean value theorem and \eqref{deriv_varphi_R} imply that 
\begin{align}\notag
&\left|\varphi_{R}(\|(y_1,z_1)\|_{X_t})-\varphi_{R}(\|(y_2,z_2)\|_{X_t})\right|\\ \label{b_2}
&\quad \leq C/R\left|\|(y_1,z_1)\|_{X_t}-\|(y_2,z_2)\|_{X_t}\right|\chi_{\left\{\|(y_2,z_2)\|_{X_t}\leq 2R\right\}}.
\end{align}
On the other hand, the property \eqref{eq:hat_rho_bound} and the H\"older inequality yield that
\begin{equation}\label{b_3}
\left\|\frac{f(y_2,y_{2,x})}{\rho}\right\|^2_{L^2(\dom)}\leq C\left\|\frac{y_2}{\hat{\rho}}\right\|^2_{H^{1}(\dom)}.
\end{equation}
In this way, inequalities \eqref{b_2}--\eqref{b_3} and the embedding $H^2(\dom)\hookrightarrow H^1(\dom)$ allows us to conclude that
\begin{align*}
I_1&\leq C/R \left\|(y_2,z_2)-(y_1,z_1)\right\|_{X_t}\left\|\frac{y_2}{\hat{\rho}}\right\|_{H^2(\dom)}\chi_{\left\{\|(y_2,z_2)\|_{X_t}\leq 2R\right\}}\\
&\leq CR\left\|(y_2,z_2)-(y_1,z_1)\right\|_{X_t}.
\end{align*}

Now, 
%
%
from estimate \eqref{eq:est_lipschitz} obtained in the Step 1 along with the definition of $\varphi_{R}$ and the assumption \eqref{b_0} yield that 
\begin{align*}
I_2&\leq C\left\|\frac{y_1-y_2}{\hat{\rho}}\right\|_{H^{1}(\dom)}\left(\left\|\frac{y_1}{\hat{\rho}}\right\|_{H^1(\dom)}+\left\|\frac{y_2}{\hat{\rho}}\right\|_{H^1(\dom)}\right)\chi_{\{\|(y_1,z_1)\|_{X_t}\leq 2R\}}\\
&\leq CR \left\|\frac{y_1-y_2}{\hat{\rho}}\right\|_{H^{1}(\dom)}.
\end{align*}
Finally, replacing the estimates for the terms $I_1$ and $I_2$ in inequality \eqref{b_1} we get
\begin{align}\label{inq:trun}
\left\|\frac{f_{R}(y_1,y_{1,x})-f_{R}(y_2,y_{2,x})}{\rho}\right\|_{L^2(\dom)}\leq CR\left(\left\|(y_2,z_2)-(y_1,z_1)\right\|_{X_t}+ \left\|\frac{y_1-y_2}{\hat{\rho}}\right\|_{H^{1}(\dom)}\right).
\end{align}

\smallskip
\textit{Step 3.} Consider the following map

\begin{equation*}
\mathcal{N}: F \in \mathcal{S} \to f_{R}(y,y_x)\in \mathcal{S},
\end{equation*}
where $y$ is taken from $(y,z)$ the solution of the system \eqref{eq:heat_source}. First, we prove that $\mathcal{N}$ is well defined. 

From inequality \eqref{inq:trun} we can deduce that 
\begin{align*}
\esp\left(\int_0^T\left\|\frac{f_R(y_1,y_{1,x})}{\rho}\right\|^2_{L^2(\dom)}\d{t}\right)\leq C^2R^2\esp\left(\int_0^T\left\|\frac{y_1}{\hat{\rho}}\right\|^2_{H^1(\dom)}\d{t}+\int_0^T\left\|(y_1,z_1)\right\|^2_{X_t}\d{t}\right). 
\end{align*}

Using that $\|\cdot\|_{X_t}\leq \|\cdot\|_{X_T}$ for all $t\in [0,T]$ and estimate \eqref{eq:reg_extra} in the above inequality we get
\begin{align}\notag 
\esp&\left(\int_0^T\left\|\frac{f_R(y_1,y_{1,x})}{\rho}\right\|^2_{L^2(\dom)}\d{t}\right)\\ \label{eq:est_wpp}
&\leq C_T\esp\left(\|y_0\|^2_{H^2(\dom)\cap H_0^1(\dom)}+\|z_0\|^2_{H_0^1(\dom)}+\int_{0}^{T}  \left\| \frac{F(t)}{\rho(t)}\right\|^2_{L^2(\dom)} \dt \right) <+\infty,
\end{align}
and therefore, we can conclude that $f_R(y,y_x)\in \mathcal{S}$.

To prove that $\mathcal{N}$ is a strictly contraction mapping we apply a similar argument as previously. It can be deduced from inequality \eqref{inq:trun}, the fact that  that $\|\cdot\|_{X_t}\leq \|\cdot\|_{X_T}$ for all $t\in [0,T]$, and estimate \eqref{eq:reg_extra}, that 
\begin{align*}
\esp\left(\int_0^T\left\|\frac{f_{R}(y_1,y_{1,x})-f_{R}(y_2,y_{2,x})}{\rho}\right\|^2_{L^2(\dom)}\d{t}\right)\leq CC_TR^2 \esp\left(\int_0^T\left\|\frac{F_1(t)-F_2(t)}{\rho}\right\|^2_{L^2(\dom)}\d{t}\right),
\end{align*}
which is equivalent to 
\begin{align*}
\|\mathcal{N}(F_1)-\mathcal{N}(F_2)\|_{\mathcal{S}}\leq C^2R^2 \|F_1-F_2\|_{\mathcal{S}}.
\end{align*}
Thus, if we choose $R$ such that $$C^2R^2<1,$$
we conclude that $\mathcal{N}$ is a strictly contraction mapping. This yields that $\mathcal N$ has a unique fixed point $F$ on $\mathcal S$ and the associated trajectory to this $F$ verifies the equation \eqref{forward_KS-H_nonlin_R} and satisfies \eqref{eq:ynulth1}.  The estimate \eqref{eq:estimatenonlinearityPropfR} follows from this fact, \eqref{eq:est_wpp} and \eqref{eq:reg_extra}. This ends the proof of \Cref{th:mainresult1}

\subsubsection{Proof of the \emph{statistical} controllability result}
Here, we present the proof of \Cref{th:mainresult2}. Markov's inequality and estimate \eqref{eq:estimatenonlinearityPropfR} coming from \Cref{th:mainresult1} imply that
\begin{align*}
\mathbb{P}\left(\|(y,z)\|^2_{X_T}\leq R^2\right)&\leq \frac{\esp\left(\|(y,z)\|^2_{X_T}\right)}{R^2}\\
&\leq \frac{e^{2C/T}}{R^2}\esp\left(\|y_0\|^2_{H^2(\dom)\cap H_0^1(\dom)}+\|z_0\|^2_{H_0^1(\dom)}\right),
\end{align*}
{whence, by hypothesis \eqref{eq:small_data}, we get}
\begin{align*}
\mathbb{P}\left(\|(y,z)\|^2_{X_T}\leq R^2\right)&<\frac{e^{2C/T}}{R^2} \delta^2<\epsilon,
\end{align*}
and this concludes the proof.

\section{Further remarks and conclusions}\label{sec:further}

We finish our work by presenting some concluding remarks regarding the controllability of fourth- and second-order stochastic parabolic systems.

\begin{enumerate}

\item \textit{On the backward nonlinear system}. In the context of \Cref{main_teo}, we have presented a stochastic model that resembles very much the stabilized Kuramoto-Sivashinsky system \eqref{KS-H}. Nevertheless, system \eqref{backward_KS-H} is linear and therefore it would be interesting to treat the nonlinear case. According to \cite[Section 4]{HSLBP20a}, it is indeed possible to adapt the \emph{statistical} null-controllability framework to backward equations. So, in this case, a careful analysis with respect to $T$ of the constant $C$ in the observability inequality \eqref{eq:obs_forward} is the first step to perform the methodology. To keep this paper at a reasonable length, we have decided not to pursue this goal but it is indeed a doable task.

\item \textit{Controlling from the heat equation}. The main results we have presented deal with the case where the systems are being controlled from the fourth-order equation and thus a natural question that arises is the possibility of controlling from the second-order parabolic equation. In this direction, it seems that \Cref{th:mainresult1,th:mainresult2} can be adapted without major modifications, since we have at hand a spectral inequality for the heat equation (see, e.g. \cite{LR95}) and the same methodology for proving \Cref{prop:obs_ineq_Xk} still applies. Nonetheless, a rigorous proof is still needed. 

However, the case of the backward equation \eqref{backward_KS-H} requires a more delicate analysis. In the deterministic setting, this question was answered in \cite{cc16} by proving a Carleman estimate for the fourth-order equation with non-homogeneous boundary conditions (see Theorem 3.5 in the aforementioned reference). The proof is based on duality arguments and requires to define the solution of the corresponding equation by transposition. In this regard, we think that the results from \cite{Gao18} to define the solution of a fourth-order stochastic equation by transposition  can be combined  with the well-known duality analysis of \cite{Liu14} to deduce the analogous result in the stochastic framework. This will be analyzed in a forthcoming paper \cite{peralta}.

\item \textit{More general coupling for the forward equation}. For the case of \Cref{th:mainresult1,th:mainresult2}, we can consider the more general coupled system 
\begin{equation}\label{eq:conc_forw}
 \begin{cases}
\d y+(y_{xxxx}+yy_x)\,\dt=(a_1 y+a_2z+\chi_{\dom_0}h)\dt+(b_1y+ b_2 z)\, \d W(t) &\text{in }Q, \\
\d z- z_{xx}\,\dt=(a_3y+a_4z)\dt+ b_3z\d W(t) &\text{in }Q, 
\end{cases}
\end{equation}
where $a_i\in L^\infty_{\mathcal F}(0,T;\mathbb R)$ and $a_3(t)\geq d_0$ or $a_3(t)\leq -d_0$ for some positive constant $d_0$ and for all $t\in[0,T]$. This comes from the fact that \Cref{prop:obs_ineq_Xk} is also valid for this system according to \cite{Liu14b}. The rest of the proof can be followed exactly.
 
It is important to mention that here we are still considering only time-dependent coefficients. As discussed in \cite[Remark 1.4]{HSLBP20a}, considering more general $x$-dependent coefficients is out of reach. 

\item \textit{A true local null-controllability resut}. We may wonder if a local null-controllability result holds for \eqref{forward_KS-H_nonlin}, that is, if there exists $\delta >0$ such that for every initial data $(y_0,z_0) \in L^2(\Omega,\fil_0;H^2(\dom)\cap H_0^1(\dom))\times L^2(\Omega,\fil_0; H_0^1(\dom)) $ verifying 
\[\norme{(y_0,z_0)}_{L^2(\Omega,\fil_0;H^2(\dom)\cap H_0^1(\dom)\times H_0^1(\dom))} \leq \delta,\]
 one can find $h \in L^2_\fil(0,T;L^2(\dom_0))$ such that the solution $(y,z) \in L_{\mathcal{F}}^2(\Omega; C([0,T];H^2(\dom)))\times  L_{\mathcal{F}}^2(\Omega; C([0,T];H_0^1(\dom)))$ of \eqref{forward_KS-H_nonlin} satisfies $y(T, \cdot) = 0$, a.s. With our current approach it seems that is not possible to get rid of the truncation procedure which is the main limitation to study the original problem. This remains as an open problem, even for the simpler nonlinear heat equation.

\item \textit{Higher dimensions}. Here we have limited ourselves to one-dimensional problems. In the case of \Cref{main_teo}, we expect that the analysis can be extended to dimensions $N\geq 2$ thanks to the very recent results of \cite{GK19} on deterministic Carleman estimates for fourth-order equations in higher dimensions. Of course, the first step is to translate the results in \cite{GK19} to the stochastic setting and this deserves further special attention. 

In the case of \Cref{th:mainresult1,th:mainresult2}, the translation of the results to higher dimensions is indeed possible but with some modifications. Indeed, the Lebeau-Robbiano technique used to prove \Cref{prop:with_cost_control} and the adaptation of the source term method are somehow independent of the dimension, but the structure of the nonlinearity should change to carry out the analysis in \Cref{sec:contr_nonlin}. As pointed out in \cite{HSLBP20a}, considering nonlinearities depending on the gradient in dimension greater than one is not so clear at the moment and only power nonlinerities of the form $f(y)=y^{p}$ for some $p>0$ only depending on $N$ can be treated. 
 
\end{enumerate}

\appendix


\section{Proof of \Cref{prop:source_stochastic}}\label{proof:source}

In the following, the constant $C_T>0$ is always of the form $\exp(C/T)$ where $C>0$ only depends on $\dom,\dom_0, b_1,b_2$, and $b_3$. In addition, $C_T$ can vary from line to line.

We define $T_k=T-\frac{T}{Q^{k}}$ for $k\geq 0$. We can easily deduce the following relation between the weights defined in \eqref{eq:defgamma}, \eqref{eq:rho0} and \eqref{eq:rho}
\begin{equation}
\label{eq:relationweights}
\rho_0(T_{k+2})=\rho(T_{k})\gamma(T_{k+2}-T_{k+1}). 
\end{equation}

We consider the equation 
\begin{equation}\label{eq:y_1}
\begin{cases}
\d y_1+(y_{1,xxxx}+F)\dt=(b_1y_1+ b_2 z_1)\, \d W(t) &\text{in } \dom \times (T_k,T_{k+1}), \\
\d z_1- z_{1,xx}\,\dt=y_1\,\dt+b_3z_1\, \d W(t) &\text{in } \dom\times (T_k,T_{k+1}), \\
y_1=y_{1,xx}=0 &\text{on } \{0,1\}\times(T_k,T_{k+1}), \\
z_1=0 &\text{on }\{0,1\}\times(T_k,T_{k+1}), \\
y_1(T_k)=z_1(T_k)=0 &\text{in } \dom. 
\end{cases}
\end{equation}

We introduce the sequence of random variables $\{(a_{k},b_{k})\}_{k\geq 0}$ such that
\begin{equation*}
(a_0,b_0) = (y_0,z_0)\ \text{and}\ (a_{k+1},b_{k+1})=(y_1(T_{k+1}),z_1(T_{k+1})).
\end{equation*}
We also consider the equation
\begin{equation}\label{eq:y_2}
\begin{cases}
\d y_2+y_{,2xxxx}\dt=\chi_{\dom_0}h_k\,\dt+(b_1y_2+ b_2 z_2)\, \d W(t) &\text{in } \dom \times (T_k,T_{k+1}), \\
\d z_2- z_{2,xx}\,\dt=y_2\,\dt+b_3z_2\, \d W(t) &\text{in }  \dom \times (T_k,T_{k+1}), \\
y_2=y_{2,xx}=0 &\text{on } \{0,1\}\times (T_k,T_{k+1}), \\
z_2=0 &\text{on } \{0,1\}\times (T_k,T_{k+1}), \\
y_2(T_k)=a_k, \quad z(T_k)=b_k &\text{in } \dom. 
\end{cases}
\end{equation}
Observe that due to the regularity of the solution of \eqref{eq:y_1}, each $a_{k},b_k$ with $k\geq 0$, is $\mathcal F_{T_k}$-measurable and belongs to $L^2(\Omega\times \dom)$. Therefore, thanks to \Cref{prop:app_well_2} the system \eqref{eq:y_2} is well posed for each $h_k\in L^2_{\mathcal F}(T_k,T_{k+1};L^2(\dom))$.

According to \Cref{prop:with_cost_control}, we can construct a control $h_k\in L^2_{\mathcal F}(T_k,T_{k+1};L^2(\dom))$ such that
\begin{equation*}
y_2(T_{k+1})=z_2(T_{k+1})=0, \quad \textnormal{a.s.}
\end{equation*}
and the following estimates holds
\begin{equation}\label{eq:cost_hk}
\esp\left(\int_{T_k}^{T_{k+1}}\!\!\!\!\int_{\dom_0}|h_k(t,x)|^2\dx\dt\right)\leq \gamma^2(T_{k+1}-T_k)\esp\left(\|a_k\|^2_{L^2(\dom)}+\|b_k\|^2_{L^2(\dom)}\right).
\end{equation}
On the other hand, appliying \Cref{prop:app_well_2} to \eqref{eq:y_1}, we have
\begin{align}\label{eq:ener_y1}
\esp&\left(\|a_{k+1}\|_{L^2(\dom)}^2+\|b_{k+1}\|_{L^2(\dom)}^2\right) \leq C_T\esp \left(\int _{T_k}^{T_{k+1}} \|F(t)\|^2_{L^2(\dom)}\dt\right) .
\end{align}

Then, combining estimates  \eqref{eq:cost_hk} and \eqref{eq:ener_y1} we get
\begin{align*}
&\esp\left(\int_{T_{k+1}}^{T_{k+2}}\!\!\!\int_{\dom_0}|h_{k+1}(t,x)|^2\dx\dt\right) \leq C_T\gamma^2\left(T_{k+2}-T_{k+1}\right)\esp\left(\int _{T_k}^{T_{k+1}} \|F(t)\|^2_{L^2(\dom)}\dt\right)
\end{align*}
and using the fact that $\rho$ is a non-increasing, deterministic function together with equality \eqref{eq:relationweights} yield
\begin{align*}
&\esp\left(\int_{T_{k+1}}^{T_{k+2}}\!\!\!\int_{\dom_0}|h_{k+1}(t,x)|^2\dx\dt\right) \leq  C_T \rho_0^2\left(T_{k+2}\right) \esp\left(\int_{T_k}^{T_{k+1}}\left\|\frac{F(t)}{\rho(t)}\right\|^2_{L^2(\dom)}\dt\right).
\end{align*}
The fact that $\rho_0$ is a non-increasing, deterministic function, implies 
\begin{align}\label{eq:norm_hk_rho}
\esp&\left(\int_{T_{k+1}}^{T_{k+2}}\!\!\!\int_{\dom_0}\left|\frac{h_{k+1}(t,x)}{\rho_0(t)}\right|^2\dx\dt\right) \leq C_T \esp\left(\int_{T_k}^{T_{k+1}}\left\|\frac{F(t)}{\rho(t)}\right\|^2_{L^2(\dom)}\dt\right).
\end{align}

Let $n\in\mathbb \N^*$. From \eqref{eq:norm_hk_rho}, we have
\begin{align} \label{eq:sum_hk} 
\esp&\left(\int_{T_1}^{T}\!\int_{\dom_0}\sum_{k=0}^{n}\mathbf{1}_{[T_{k+1},T_{k+2})}(t)\left|\frac{h_{k+1}}{\rho_0}\right|^2\dx\dt\right) \leq C_T\esp\left(\int_{0}^{T}\sum_{k=0}^n\mathbf{1}_{[T_{k},T_{k+1})}(t)\left\|\frac{F(t)}{\rho(t)}\right\|^2_{L^2(\dom)}\dt\right).
\end{align}
From \eqref{eq:cost_hk} at $k=0$, recalling that $(a_0,b_0)=(y_0,z_0)$ and using the properties of $\rho_0$, we get
\begin{equation}\label{eq:est_h0}
\esp\left(\int_{0}^{T_1}\!\!\!\int_{\dom_0}\left|\frac{h_0}{\rho_0}\right|^2\dx\dt\right)\leq \frac{\gamma^2\left(T_1\right)}{\rho_0^2(T_1)}\esp\left(\|y_0\|^2_{L^2(\dom)}+\|z_0\|^2_{L^2(\dom)}\right).
\end{equation}
Putting together \eqref{eq:sum_hk} and \eqref{eq:est_h0} yields the existence of a constant $C_T>0$ independent of $n$ such that
\begin{align*} 
\esp&\left(\int_{0}^{T_1}\!\!\!\int_{\dom_0}\left|\frac{h_0}{\rho_0}\right|^2\dx\dt\right)+\esp\left(\int_{T_1}^{T}\!\int_{\dom_0}\sum_{k=0}^{n}\mathbf{1}_{[T_{k+1},T_{k+2})}(t)\left|\frac{h_{k+1}}{\rho_0}\right|^2\dx\dt\right) \\
&\leq C_T\esp\left(\|y_0\|^2_{L^2(\dom)}+\|z_0\|^2_{L^2(\dom)}+\int_{0}^{T}\sum_{k=0}^n\mathbf{1}_{[T_{k},T_{k+1})}(t)\left\|\frac{F(t)}{\rho(t)}\right\|^2_{L^2(\dom)}\dt\right).
\end{align*}
Finally, we set $h:=\sum_{k=0}^\infty h_k$ thus, using Lebesgue's convergence theorem, we obtain
\begin{align}\label{est:control_rho_0} 
\esp&\left(\int_{0}^{T}\!\!\!\int_{\dom_0}\left|\frac{h}{\rho_0}\right|^2\dx\dt\right) \leq C_T\esp\left(\|y_0\|^2_{L^2(\dom)}+\|z_0\|^2_{L^2(\dom)}+\int_{0}^{T} \left\|\frac{F(t)}{\rho(t)}\right\|^2_{L^2(\dom)}\right).
\end{align}

Applying It\^{o}'s rule to $y:=y_1+y_2$ and $z:=z_1+z_2$ for $t\in[T_k,T_{k+1})$, we get
\begin{align}
\label{eq:y}
\begin{cases}
\d y+(y_{xxxx}+F)\dt=\chi_{\dom_0}h_k\,\dt+(b_1y+ b_2 z)\, \d W(t) &\text{in } \dom \times (T_k,T_{k+1}), \\
\d z- z_{xx}\,\dt=y\,\dt+b_3z\, \d W(t) &\text{in } \dom \times (T_k,T_{k+1}), \\
y=y_{xx}=0 &\text{on }\{0,1\}\times(0,T), \\
z=0 &\text{on }\{0,1\}\times(0,T), \\
y(T_k)=a_k, \quad z(T_k)=b_k &\text{in } \dom. 
\end{cases}
\end{align}
Note that by construction $(y,z)$ is continuous at $T_k$ a.s., for all $k\geq 0$, therefore by using \eqref{eq:y}, $(y,z)$ is a solution to \eqref{eq:heat_source}.

Moreover, applying \Cref{prop:app_well_2} to \eqref{eq:y} for $k\geq 1$, we have
\begin{align*}
\esp&\left(\sup_{T_k\leq t\leq T_{k+1}} \|y(t)\|_{L^2(\dom)}^2+\sup_{T_k\leq t\leq T_{k+1}} \|z(t)\|_{L^2(\dom)}^2\right) \\ 
&\leq C_T\esp \left(\|a_k\|^2_{L^2(\dom)}+ \|b_k\|^2_{L^2(\dom)} +\int _{T_{k}}^{T_{k+1}}\left[ \|\chi_{\dom_0}h_{k}(t)\|^2_{L^2(\dom)}+\|F(t)\|^2_{L^2(\dom)}\right]\dt\right).
\end{align*}
Then using inequalities \eqref{eq:cost_hk} and \eqref{eq:ener_y1} to estimate in the above equation yields
\begin{align*}
\esp&\left(\sup_{T_k\leq t\leq T_{k+1}} \|y(t)\|_{L^2(\dom)}^2+\sup_{T_k\leq t\leq T_{k+1}} \|z(t)\|_{L^2(\dom)}^2\right) \\ 
& \leq C_T \gamma^2(T_{k+1}-T_k) \esp \left(\int _{T_{k-1}}^{T_{k+1}} \|F(t)\|^2_{L^2(\dom)}\dt\right).
\end{align*}

The identity \eqref{eq:relationweights} allows us to get
\begin{align*}
&\esp\left(\sup_{T_k\leq t\leq T_{k+1}} \|y(t)\|_{L^2(\dom)}^2+\sup_{T_k\leq t\leq T_{k+1}} \|z(t)\|_{L^2(\dom)}^2\right) \leq C_T\rho_0^2(T_{k+1}) \esp \left(\int _{T_{k-1}}^{T_{k+1}} \left\| \frac{F(t)}{\rho(t)}\right\|^2_{L^2(\dom)}\dt\right),
\end{align*}
so by using that $\rho_0$ is non-increasing, we have
\begin{align}
\esp\left(\sup_{T_k\leq t\leq T_{k+1}} \left\|\frac{y(t)}{\rho_0(t)}\right\|_{L^2(\dom)}^2+\sup_{T_k\leq t\leq T_{k+1}} \left\|\frac{z(t)}{\rho_0(t)}\right\|_{L^2(\dom)}^2\right) \leq C_T \esp \left(\int _{T_{k-1}}^{T_{k+1}} \left\| \frac{F(t)}{\rho(t)}\right\|^2_{L^2(\dom)}\dt\right). \label{eq:sup_y_rho_0}
\end{align}
Moreover, arguing as before, it is not difficult to establish that 
\begin{align}
\notag
\esp &  \left(\sup_{0\leq t\leq T_1}\left\|\frac{y(t)}{\rho_0(t)}\right\|^2_{L^2(\dom)}+\sup_{0\leq t\leq T_1}\left\|\frac{z(t)}{\rho_0(t)}\right\|^2_{L^2(\dom)}\right) \\ \label{eq:estimateynear0}
&\leq C_T \left(\|y_0\|^2_{L^2(\dom)}+\|z_0\|^2_{L^2(\dom)}+\int_{0}^{T_1}  \left\| \frac{F(t)}{\rho(t)}\right\|^2_{L^2(\dom)}\dt \right).
\end{align}

Let $n\in\mathbb N^*$. From inequalities \eqref{eq:sup_y_rho_0} and \eqref{eq:estimateynear0}, we have
\begin{align}\notag 
\esp & \left(\sup_{0\leq t\leq T_1}\left\|\frac{y(t)}{\rho_0(t)}\right\|^2_{L^2(\dom)}+\sup_{0\leq t\leq T_1}\left\|\frac{z(t)}{\rho_0(t)}\right\|_{L^2(\dom)}^2\right) \\ \notag
&+\sum_{k=1}^{n}\esp\left(\sup_{T_{k}\leq t\leq T_{k+1}}\left\|\frac{y(t)}{\rho(t)}\right\|^2_{L^2(\dom)}+\sup_{T_{k}\leq t\leq T_{k+1}}\left\|\frac{z(t)}{\rho(t)}\right\|^2_{L^2(\dom)}\right) \\ \label{eq:sup_y_n}
&\leq C_T\esp\left(\|y_0\|^2_{L^2(\dom)}+\|z_0\|^2_{L^2(\dom)}+\sum_{k=1}^{n}\int_{0}^{T}\mathbf{1}_{[T_{k-1},T_{k+1})}  \left\| \frac{F(t)}{\rho(t)}\right\|^2_{L^2(\dom)} \right)
\end{align}
where $C_T>0$ is uniform with respect to $n$. Letting $n\to \infty$ in \eqref{eq:sup_y_n} yields
\begin{align} \notag
\esp &  \left(\sup_{0\leq t\leq T }\left\|\frac{y(t)}{\rho_0(t)}\right\|^2_{L^2(\dom)}+\sup_{0\leq t\leq T }\left\|\frac{z(t)}{\rho_0(t)}\right\|^2_{L^2(\dom)}\right) \\
\label{eq:sup_y_limit_aux}
&\leq  C_T\esp\left(\|y_0\|^2_{L^2(\dom)}+\|z_0\|^2_{L^2(\dom)}+\int_{0}^{T}  \left\| \frac{F(t)}{\rho(t)}\right\|^2_{L^2(\dom)}\dt \right).
\end{align}
Finally, combining \eqref{est:control_rho_0} and \eqref{eq:sup_y_limit_aux} gives the desired result. This concludes the proof.
%

\section{Well-posedness results}

We devote this section to present some general results and make some comments about the well-posedness of systems \eqref{backward_KS-H}, \eqref{eq:adjoint_KSS} and systems \eqref{eq:heat_source}, \eqref{eq:system_back_tau}. For conciseness, we assume that the coefficients $d_i$ and $b_i$ have the same regularity as in \Cref{main_teo} and \Cref{th:mainresult1,th:mainresult2}, respectively.

\subsection{Clamped boundary conditions}

We present here the well-posedness results for the case of \textit{clamped} boundary conditions for the fourth-order operator, i.e., $y=y_x=0$ on $\Sigma$. We begin with the forward system. 

\begin{prop}\label{prop:app_well_1}
Assume that $u_0,v_0\in L^2(\Omega,\mathcal F_0;L^2(\dom))$ and $f_i,g_i\in L^2_{\mathcal F}(0,T;L^2(\dom))$, $i=1,2$. Then, the system
\begin{equation}\label{app:forward_system}
\begin{cases}
\d u+(\gamma u_{xxxx}+u_{xxx}+u_{xx})\dt=(f_1+v_x )dt+(g_1+d_1 u)\, \d W(t) &\text{in }Q, \\
\d v-\Gamma v_{xx}\dt=(f_2+v_x+u_x)\dt+(g_2+d_2u+d_3 v) \d W(t) &\text{in }Q, \\
u=u_x=0 &\text{on }\Sigma, \\
v=0 &\text{on }\Sigma, \\
u(0)=u_0, \quad v(0)=v_0 &\text{in } \dom,
\end{cases}
\end{equation}
has a unique solution $(u,v)\in L^2_{\mathcal F}(0,T;H_0^2(\dom)\times H_0^{1}(\dom)) \bigcap L^2\left(\Omega; C([0,T];L^2(\dom)^2)\right)$. Moreover, there exists some $C>0$ only depending on $T$, $\Gamma$, $\gamma$ and $d_i$, $i=1,2,3$, such that
\begin{align*}
\esp\left(\sup_{0\leq t \leq T} \|u\|_{L^2(\dom)}^2+ \sup_{0\leq t \leq T} \|v\|^2_{L^2(\dom)}\right)+\esp\left(\int_{0}^{T}\|u(t)\|^2_{H_0^2(\dom)}\dt+\int_{0}^{T}\|v(t)\|^2_{H_0^1(\dom)}\dt\right) \\
\leq C\esp \left(\|u_0\|^2_{L^2(\dom)}+\|v_0\|_{L^2(\dom)}^2+\sum_{i=1}^{2}\left\{\int_{0}^{T}\|f_i(t)\|^2_{L^2(\dom)}\dt
+\int_{0}^{T}\|g_i(t)\|^2_{L^2(\dom)}\dt
\right\}\right).
\end{align*}
\end{prop}

The proof of this result can be done by applying well-posedness results to each equation in the system and some integration by parts. Seeing individually, the existence and uniqueness of the solutions to the stochastic parabolic equation have been studied in \cite{krylov} (see \cite[Proposition 2.1]{zhou92} for a more accessible reference), while the analysis of the fourth-order PDE has been made in \cite{Gao15}. 

We present a general result for the backward system \eqref{backward_KS-H}. 

\begin{prop}\label{prop:app_well_2}
Assume that $y_T,z_T\in L^2(\Omega,\mathcal F_T;L^2(\dom))$ and $F_i\in L^2_{\mathcal F}(0,T;L^2(\dom))$, $i=1,2$. Then, the system
\begin{equation}\label{app:backward_system}
\begin{cases}
\d y-(\gamma y_{xxxx}-y_{xxx}+y_{xx})\dt=(z_x-d_1 Y-d_2 Z+F_1)\dt+ Y\d W(t) &\text{in }Q, \\
\d z+\Gamma z_{xx}\dt=(z_x+y_x-d_3Z+F_2)\dt+Z\d W(t) &\text{in }Q, \\
y=y_x=0 &\text{on }\Sigma, \\
z=0 &\text{on }\Sigma, \\
y(T)=y_T, \quad z(T)=z_T &\text{in } \dom,
\end{cases}
\end{equation}
has a unique solution $(y,z,Y,Z)\in \left[L^2_{\mathcal F}(0,T;H_0^2(\dom)\times H_0^{1}(\dom)) \bigcap L^2\left(\Omega; C([0,T];L^2(\dom)^2)\right)\right]\times L^2_{\mathcal F}(0,T;L^2(\dom))^2$. Moreover, there exists some $C>0$ only depending on $T$, $\Gamma$, $\gamma$ and $d_i$, $i=1,2,3$, such that
\begin{align*}
\esp&\left(\sup_{0\leq t \leq T}\|y\|_{L^2(\dom)}^2+\sup_{0\leq t \leq T}\|z\|^2_{L^2(\dom)}\right)+\esp\left(\int_{0}^{T}\|y(t)\|^2_{H_0^2(\dom)}\dt+\int_{0}^{T}\|z(t)\|^2_{H_0^1(\dom)}\dt\right) \\
&+\esp \left(\int_{0}^T \|Y(t)\|_{L^2(\dom)}^2\dt+\int_{0}^T \|Z(t)\|_{L^2(\dom)}^2\dt\right) \\
&\quad \leq C\esp \left(\|y_T\|^2_{L^2(\dom)}+\|z_T\|_{L^2(\dom)}^2+\int_{0}^{T}\|F_1(t)\|^2_{L^2(\dom)}\dt
+\int_{0}^{T}\|F_2(t)\|^2_{L^2(\dom)}\dt
\right).
\end{align*}
\end{prop}

The proof is totally analogous to the one of \Cref{prop:app_well_1}. We just need to change the existence and uniqueness result for each individual equation, that is, we need to consider \cite[Theorem 3.1]{zhou92} for the parabolic equation and \cite[Proposition 2.4]{Gao15} for the fourth-order one.

\subsection{Hinged boundary conditions}

Here, we present existence and uniqueness results for the case of systems with \textit{hinged} boundary conditions, that is, when we impose $y=y_{xx}=0$ on $\Sigma$ to the fourth-order operator. 

In what follows, we assume that $a_{i}\in L^\infty_{\fil}(0,T;\R)$, $i=1,\dots,4$ and that $T\in(0,1)$. We have the following result for the forward system.

\begin{prop}\label{prop:app_well_2}
\begin{enumerate}
\item[a)] Assume that $y_0,z_0\in L^2(\Omega,\mathcal F_0;L^2(\dom))$ and $f_i,g_i\in L^2_{\mathcal F}(0,T;L^2(\dom))$, $i=1,2$. Then, the system
\begin{equation}\label{app:forward_system_simp}
\begin{cases}
\d y+y_{xxxx}\dt=(f_1+a_1 y+ a_2 z)dt+(g_1+b_1 y+b_2 z)\, \d W(t) &\text{in }Q, \\
\d z- z_{xx}\dt=(f_2+a_3y+a_4z)\dt+(g_2+b_3 z) \d W(t) &\text{in }Q, \\
y=y_{xx}=0 &\text{on }\Sigma, \\
z=0 &\text{on }\Sigma, \\
y(0)=y_0, \quad z(0)=z_0 &\text{in } \dom,
\end{cases}
\end{equation}
has a unique solution $(y,z)\in L^2_{\mathcal F}(0,T;H^2(\dom)\times H_0^{1}(\dom)) \bigcap L^2\left(\Omega; C([0,T];L^2(\dom)^2)\right)$. Moreover, there exists some $C>0$ only depending on  $a_j$ and $b_i$, such that
\begin{align}\notag
\esp\left(\sup_{0\leq t \leq T} \|y(t)\|_{L^2(\dom)}^2+\sup_{0\leq t \leq T}\|z(t)\|^2_{L^2(\dom)}\right)+\esp\left(\int_{0}^{T}\|y(t)\|^2_{H^2(\dom)}\dt+\int_{0}^{T}\|z(t)\|^2_{H_0^1(\dom)}\dt\right) \\ \label{eq:app_reg_a}
\leq C\esp \left(\|y_0\|^2_{L^2(\dom)}+\|z_0\|_{L^2(\dom)}^2+\sum_{i=1}^{2}\left\{\int_{0}^{T}\|f_i(t)\|^2_{L^2(\dom)}\dt
+\int_{0}^{T}\|g_i(t)\|^2_{L^2(\dom)}\dt
\right\}\right).
\end{align}
\end{enumerate}
\item[b)] Assume that $(y_0,z_0)\in L^2(\Omega,\mathcal F_0;H^2(\dom)\cap H_0^1(\dom))\times L^2(\Omega,\mathcal F_0;H_0^1(\dom))$ and $f_i\in L^2_{\mathcal F}(0,T;L^2(\dom))$, $g_i\in L^2_{\mathcal F}(0,T; H^{3-i}(\dom))$, $i=1,2$. Then, the system \eqref{app:forward_system_simp}
has a unique solution $$(y,z)\in L^2_{\mathcal F}(0,T;H^4(\dom)\times H^2(\dom)) \bigcap L^2\left(\Omega; C([0,T];H^2(\dom)\times H_0^1(\dom))\right).$$
Moreover, there exists some $C>0$ only depending on $a_j$ and $b_i$, such that
\begin{align*}
\esp\left(\sup_{0\leq t \leq T} \|y(t)\|_{H^2(\dom)}^2+\sup_{0\leq t \leq T}\|z(t)\|^2_{H_0^1(\dom)}\right)+\esp\left(\int_{0}^{T}\|y(t)\|^2_{H^4(\dom)}\dt+\int_{0}^{T}\|z(t)\|^2_{H^2(\dom)}\dt\right) \\
\leq C\esp \left(\|y_0\|^2_{H^2(\dom)\cap H_0^1(\dom)}+\|z_0\|_{H^1_0(\dom)}^2+\sum_{i=1}^{2}\left\{\int_{0}^{T}\|f_i(t)\|^2_{L^2(\dom)}\dt
+\int_{0}^{T}\|g_i(t)\|^2_{H^{3-i}(\dom)}\dt
\right\}\right).
\end{align*}
\end{prop}

The proof of this result can be obtained in two steps. For a), we can adapt \cite[Proposition 2.3]{Gao15} to get a well-posedness result for the fourth-order equation with hinged boundary conditions. Actually, we just need to change the eigenvalue problem and the basis employed for the Galerkin method in that work. Once this is done, we can put together the obtained result with classical regularity estimates (see \cite[Proposition 2.1]{zhou92}) for the heat equation to deduce \eqref{eq:app_reg_a}. Point b) comes by arguing similar to \cite[Proposition 2.1]{gao2017} to deduce an energy estimate for the fourth-order equation with more regular data.

We conclude this section by giving the following result for the corresponding backward system. 

\begin{prop}\label{prop:back_uxx} Assume that $u_T,v_T\in L^2(\Omega,\mathcal F_T; L^2(\dom))$ and $F_i\in L^2_{\fil}(0,T;L^2(\dom))$. Then, the system
\begin{equation}\label{eq:system_back_tau_app}
\begin{cases}
\d u-u_{xxxx}\dt=-(a_1 u+a_3v+F_1+b_1\ov{u})\dt+\ov{u}\,\d W(t) &\text{in } Q, \\
\d v+v_{xx}\dt=-(a_2 u + a_4 v+F_2+b_2\ov{u}+b_3\ov{v})\dt+\ov{v}\,\d W(t) &\text{in } Q, \\
u=u_{xx}=0 &\text{on }\Sigma, \\
v=0 &\text{on } \Sigma, \\
u(T)=u_T, \quad v(T)=v_T &\text{in } \dom,
\end{cases}
\end{equation}
has a unique solution 
\begin{equation*}
(u,v,\ov{u},\ov{v})\in \left[L^2_{\mathcal F}(0,T;H^2(\dom)\times H_0^{1}(\dom)) \bigcap L^2\left(\Omega; C([0,T];L^2(\dom)^2)\right)\right]\times L^2_{\mathcal F}(0,T;L^2(\dom))^2.
\end{equation*}

Moreover, there exists some $C>0$ only depending on $a_j$ and $b_i$, such that
\begin{align*}
\esp&\left(\sup_{0\leq t \leq T}\|u\|_{L^2(\dom)}^2+\sup_{0\leq t \leq T}\|v\|^2_{L^2(\dom)}\right)+\esp\left(\int_{0}^{T}\|u(t)\|^2_{H^2(\dom)}\dt+\int_{0}^{T}\|v(t)\|^2_{H_0^1(\dom)}\dt\right) \\
&+\esp \left(\int_{0}^T \|\ov{u}(t)\|_{L^2(\dom)}^2\dt+\int_{0}^T \|\ov{v}(t)\|_{L^2(\dom)}^2\dt\right) \\
&\quad \leq C\esp \left(\|u_T\|^2_{L^2(\dom)}+\|v_T\|_{L^2(\dom)}^2+\int_{0}^{T}\|F_1(t)\|^2_{L^2(\dom)}\dt
+\int_{0}^{T}\|F_2(t)\|^2_{L^2(\dom)}\dt
\right).
\end{align*}
\end{prop}

\section{Sketch of the proof of Lemma \ref{lem:local_u3x}}\label{proof_lemma_u3x}
Since most of the arguments are similar to those in Step 3 of the proof of Theorem \ref{thm:car_estimate_KSS}, we proceed briefly. Let us consider an open set $\dom_{2}$ such that $\dom_{3}\subset\subset\dom_{2}\subset\subset\dom_0$ and take $\hat{\eta}\in C^\infty_0(\dom_{2})$ satisfying $\hat\eta\equiv1$ in $\dom_3$. Using It\^{o}'s formula, we compute $\d(\hat{\zeta}u u_{xx})$, where $\hat{\zeta}:=\hat{\eta}\lambda^5\phi_m^5\theta^2$. After a long, but straightforward computation we get
\begin{align}\notag
2\gamma \esp\left(\int_{Q}\hat\zeta|u_{xxx}|^2\dx\dt\right)=&-\gamma\esp\left(\int_Q\left[3\hat\zeta_{xx}u_xu_{xxx}+\hat{\zeta}_{xxx}uu_{xxx}-2\hat{\zeta}_{xx}|u_{xx}|^2\right]\dx\dt\right) \\ \notag
&+\esp\left(\int_{Q}\left[2\hat{\zeta}_xu_xu_{xxx}+\hat{\zeta}_{xx}uu_{xxx}+\hat{\zeta}_{xx}uu_{xx}-\hat{\zeta}_{xx}|u_x|^2\right]\dx\dt\right) \\ \notag
&+\esp\left(\int_{Q}\left[\hat{\zeta}u_x v_{xx}+\hat{\zeta}_x u v_{xx}-\hat{\zeta}_t u u_{xx}-\hat{\zeta}_x|u_{xx}|^2\right]\dx\dt\right) \\ \notag
&+\esp\left(\int_{Q}\left[2\hat{\zeta}|u_{xx}|^2-\hat{\zeta}u_{xx}v_x\right]\dx\dt\right) \\ \label{eq:est_uxxx}
&+\esp\left(\int_{Q}\left[\hat\zeta|(d_1u)_x|^2-\frac{1}{2}\hat{\zeta}_{xx}|d_1 u|^2\right]\dx\dt\right).
\end{align}

Using the definition of $\phi_m$ and $\theta$, we can see that
\begin{gather*}
|\partial_x^{i}(\theta^2\phi_m^p)|\leq C_i \lambda^{i}\phi_m^{p+i}\theta^2, \quad i=1,2,3, \\
|\partial_t(\theta^2\phi_m^p)|\leq C\lambda\phi_m^{p+1+\frac{1}{m}}\theta^2, \quad \forall p\in\mathbb N^*,
\end{gather*}
from which we deduce
\begin{equation}\label{eq:prop_zetahat}
\begin{split}
|\hat{\zeta}_t|&\leq C\lambda^6\phi_m^{6+\frac{1}{m}}\theta^2\eta, \\
|\partial_x^{i}(\hat{\zeta})|&\leq C_i\lambda^{5+i}\phi_m^{5+i}\theta^2\sum_{j=0}^{i}(\partial_x^j\eta), \quad i=1,2,3,\\
\end{split}
\end{equation}
for all $(x,t)\in \dom_{2}\times(0,T)$. 

Using Cauchy-Schwarz and Young inequalities together with \eqref{eq:prop_zetahat} and taking into account the properties of the function $\hat\eta$, we can obtain from \eqref{eq:est_uxxx} the following estimate
\begin{align*}\notag 
\esp&\left(\int_{Q_{\dom_3}}\lambda^5\phi_m^5\theta^2|u_{xxx}|^2\dx\dt\right)\\
&\leq 4\epsilon\esp\left(\int_{Q}\lambda\phi_m|u_{xxx}|^2\dx\dt\right)+2\delta\esp\left(\int_{Q}\lambda\phi_m\theta^2|v_{xx}|^2\dx\dt\right) \\
&+\rho\esp\left(\int_{Q}\lambda^3\phi_m^3\theta^2|v_{x}|^2\dx\dt\right)+C\esp\left(\int_{Q_{\dom_{2}}}\lambda^{13}\phi_m^{13}\theta^2|u_x|^2\dx\dt\right) \\ \notag
&+C\esp\left(\int_{Q_{\dom_{2}}}\lambda^{15}\phi_m^{15}\theta^2|u|^2\dx\dt\right)+C\esp\left(\int_{Q_{\dom_{2}}}\lambda^{7}\phi_m^{7}\theta^2|u_{xx}|^2\dx\dt\right),
\end{align*}
for any positive constants $\epsilon,\delta,\rho$ and where $C>0$ depends on $\|d_1\|_{L^\infty_{\mathcal F}(0,T;W^{2,\infty}(\dom))}$. This concludes the proof.

\bibliographystyle{plain}
\small{\bibliography{bibstoch}}

\end{document}